\documentclass[leqno,11pt]{amsart}
\usepackage{amsmath,amsthm, amsfonts,amssymb}
\usepackage{enumerate}
\usepackage{hyperref}
\usepackage[T1]{fontenc}
\newtheorem{thm}{Theorem}[section]
\newtheorem*{thm*}{Theorem}
\newtheorem{lem}[thm]{Lemma}

\newtheorem{cor}[thm]{Corollary}
\newtheorem{pro}[thm]{Proposition}

\theoremstyle{definition}

\theoremstyle{remark}
\newtheorem*{remark}{Remark}
\newtheorem*{remark1}{Remark 1}
\newtheorem*{remark2}{Remark 2}

\newcommand{\tpsi}{\tilde{\psi}}
\newcommand{\tfi}{\tilde{\varphi}}
\newcommand{\la}{\lambda}
\newcommand{\TT}{\mathbb{T}}
\newcommand{\mS}{\mathcal{S}}

\newcommand{\mD}{\mathcal{D}}
\newcommand{\mA}{\mathcal{A}}

\newcommand{\mFZd}{\mathcal{F}_{\mathbb{Z}^d}}
\newcommand{\mFZ}{\mathcal{F}_{\mathbb{Z}}}

\newcommand{\bR}{\mathbb{R}}
\def\J{\mathcal J} 	
\def\P{\mathcal P} 	
\newcommand{\bC}{\mathbb{C}}

\newcommand{\Rd}{\mathbb{R}^d}

\newcommand{\ab}{\alpha,\beta}

\newcommand{\dLZd}{\Delta_{\mathbb{Z}^d}}
\newcommand{\dLZ}{\Delta_{\mathbb{Z}}}
\newcommand{\ka}{\kappa}
\newcommand{\dc}{*_{\kappa}}
\newcommand{\bN}{\mathbb{N}}
\newcommand{\bZ}{\mathbb{Z}}

\newcommand{\e}{\varepsilon}
\newcommand{\mP}{\Pi}
\newcommand{\Do}{\delta}
\newcommand{\lan}{\left\langle}
\newcommand{\ran}{\right\rangle}

\newcommand{\NN}{{\mathbb{N}}}
\newcommand{\ZZ}{{\mathbb{Z}}}
\newcommand{\RR}{{\mathbb{R}}}

\DeclareMathOperator{\Real}{Re}

\DeclareMathOperator{\Dom}{Dom}

\DeclareMathOperator{\supp}{supp}
\DeclareMathOperator{\Sin}{Sin}

\title[Bilinear multipliers via a functional calculus]{Approaching bilinear multipliers via a functional calculus}
\author{B\l a\.{z}ej Wr\'obel}
\address{B\l a\.zej Wr\'obel:
	Mathematical Institute
	\\ 	Universit\"at Bonn\\
	Endenicher Allee 60\\ D--53115 Bonn\\ Germany
	\newline \&
	Instytut Matematyczny, Uniwersytet Wroc\l awski,
	pl. Grunwaldzki 2/4, 50-384 Wroc\l aw, Poland
	\hfill\break
	blazej.wrobel@math.uni.wroc.pl}
\email{\ \newline
blazej.wrobel@math.uni.wroc.pl}
 
\begin{document}

 \begin{abstract}
We propose a framework for bilinear multiplier operators defined via the (bivariate) spectral theorem. Under this framework we prove Coifman-Meyer type multiplier theorems and fractional Leibniz rules. Our theory applies to bilinear multipliers associated with the discrete Laplacian on $\mathbb{Z}^d,$ general bi-radial bilinear Dunkl multipliers, and to bilinear multipliers associated with the Jacobi expansions.
 \end{abstract}
  \maketitle
 \numberwithin{equation}{section}
 \section{Introduction}
 
The theory of spectral multipliers is now a well established and vast branch of linear harmonic analysis. Its origins lie in trying to extend the Fourier multiplier operators on $\RR$  given by
\begin{equation*}
f\mapsto \frac{1}{2\pi}\int_{\mathbb{R}}m(\xi)\hat{f}(\xi)e^{ix\xi}\, d\xi,\qquad x\in \RR,
\end{equation*}
to other settings. Here  $m$ is a bounded  function on $\RR$ while  $\hat{f}(\xi)=\int_{\RR}f(x)e^{-ix \xi}\,dx,$ $\xi\in \RR.$ For a self-adjoint operator $L$ its spectral multipliers are the operators $m(L)$ defined by the spectral theorem. In the Fourier case $L$ is merely $i\frac{d}{dx}.$ As in the Fourier case the boundedness of $m(L)$ on $L^2$ is equivalent with the boundedness of $m.$ The main task in the theory of spectral multipliers is to extend the boundedness of $m(L)$  to $L^p,$ for some $1<p<\infty,$ $p\neq 2.$
 
The bilinear multipliers for the Fourier transform are the operators
\begin{equation}
\label{eq:Fourdef} F_m(f_1,f_2)(x)=\frac{1}{4\pi^2}\iint_{\mathbb{R}^2}m(\xi_1,\xi_2)\hat{f_1}(\xi_1)\hat{f_2}(\xi_2)e^{ix(\xi_1+\xi_2)}\, d\xi,\qquad x\in \RR,
\end{equation} 
with $m\colon \RR^2\to \mathbb{C}$ being a bounded function.  As far as we know, in the bilinear case, there has been no systematic approach to extend the operators $F_m$ outside of the Fourier transform setting. The main idea behind the creation of this paper is to provide a theory for bilinear multipliers defined by the (bivariate) spectral theorem that parallels the correspondence between the linear Fourier multipliers and spectral multipliers. 

Our starting point is the observation that \eqref{eq:Fourdef} may be rephrased as
$$F_m(f_1,f_2)(x)=m(i\partial_1,i\partial_2)(f_1\otimes f_2)(x,x),\qquad x\in \RR;$$
here $\partial_1,$ $\partial_2$ denote the partial derivatives, while $m(i\partial_1,i\partial_2)$ is defined by the bi-variate spectral theorem. Note that $\partial_1=\partial\otimes I$ and $\partial_2=I\otimes \partial,$ where $\partial$ denotes the derivative on $\RR,$ while $I$ is the identity operator. We investigate the possibility of replacing $i\partial_1$ and $i\partial_2$ by some other operators $L_1=L\otimes I$ and $L_2=I\otimes L$. The bilinear multipliers we consider are of the form
\begin{equation}
\label{eq:-1}
B_m(f_1,f_2)(x)=m(L_1,L_2)(f_1\otimes f_2)(x,x),\qquad x\in X.\end{equation}
Here $L$ is a self-adjoint non-negative operator  on $L^2(X,\nu)$, and $m(L_1,L_2)$ is defined by the bi-variate spectral theorem. We also assume that $L$ is injective on its domain, and that the contractivity condition \eqref{eq:Con}  (see p.\ \pageref{p:CT}) and the well definiteness condition \eqref{eq:WD} (see p.\ \pageref{p:WD}) are satisfied. These assumptions should be regarded as technical ones. The main assumptions on $L$ that are in force in this paper are the existence of a Mikhlin-H\"ormander functional calculus  \eqref{eq:MH}, see p. \pageref{p:asum}, together with a product formula for the spectral multipliers of $L$, see \eqref{eq:PF} on p. \pageref{p:asumPF}. 


There are two main goals of our paper. Firstly, we
would like to prove Coifman-Meyer type multiplier theorems outside of the Fourier transform setting. Secondly, we would like to apply these results to obtain fractional Leibniz rules.

The classical Coifman-Meyer multiplier theorem \cite{CM1} says that the Mikhlin-H\"ormander condition $\sup_{\xi \in \RR^2}|\xi|^{\alpha_1+\alpha_2}|\partial^{\alpha}m(\xi)|\leq C_{\alpha},$ $\alpha\in \NN^2,$ implies the boundedness of $F_m$ from $L^{p_1}\times L^{p_2}$ to $L^p,$ $1/p=1/p_1+1/p_2,$ $p_1>1,$ $p_2>1,$ $p>1/2.$ This was proved by Coifman and Meyer for $p>1,$ while for $p>1/2$ it is due to Grafakos and Torres \cite{GraTor1} and Kenig and Stein \cite{KenSte1}. There are also Coifman-Meyer type multiplier theorems which are known in settings other than the Fourier transform. For bilinear multipliers on the torus  a theorem of Coifman-Meyer type may be deduced from  Fan and Sato \cite[Theorems 1-3]{FanSat1}. Similarly, for bilinear multipliers on the integers such a theorem follows from Blasco \cite[Theorem 3.4]{Bla1}.  Next, in the product Dunkl setting a Coifman-Meyer type multiplier theorem was proved by Amri, Gasmi, and Sifi \cite{AmGasSif1}. 

The main result of this paper is the following generalized Coifman-Meyer type theorem.
\begin{thm*}[Theorem \ref{thm:genCoif-Mey}]
	Let $m\colon (0,\infty)^2\to \mathbb{C}$ satisfy the H\"ormander's  condition $$|\la|^{\alpha_1+\alpha_2}\,|\partial^{\alpha}m(\la)|\leq C_{\alpha},\qquad \la \in (0,\infty)^2,$$
	for sufficiently many multi-indices $\alpha\in \NN^2.$ 
	Then $B_m$ given by \eqref{eq:-1} is bounded from $L^{p_1}(X)\times L^{p_2}(X)$ to $L^p(X),$ where $1/p_1+1/p_2=1/p,$ with $p_1,p_2,p>1.$ 
\end{thm*}
\noindent Theorem \ref{thm:genCoif-Mey} is formally stated and proved in Section \ref{sec:Gen}. The main difficulty in obtaining the theorem lies in finding an appropriate proof of the classical Coifman-Meyer multiplier theorem, which is prone to modifications towards our setting. The proof we present in Section \ref{sec:Gen} follows the scheme by Muscalu and Schlag \cite[pp. 67-71]{Musc_Schlag}. An important ingredient in our proof is a spectrally defined Littlewood-Paley theory. For this method to work the assumption \eqref{eq:PF} is very useful.  It might be interesting to try to replace \eqref{eq:PF} with a less rigid condition.

An application of Theorem \ref{thm:genCoif-Mey} provides  Coifman-Meyer type multiplier results for bilinear multipliers given by \eqref{eq:-1} in three cases different than the Fourier transform setting. In Theorem \ref{thm:discLap_Coif-Mey} we treat bilinear multipliers for $L$ being the discrete Laplacian on $\ZZ^d.$ This is close to  \cite[Theorem 3.4]{Bla1}, however our results here are of a different kind. In Theorem \ref{thm:DunklLap_Coif-Mey} we consider  bi-radial bilinear Dunkl multipliers, here $L$ is the general Dunkl Laplacian. In Corollary \ref{cor:AmGasSif} we also reprove  \cite[Theorem 4.1]{AmGasSif1}. Finally, in Theorem \ref{thm:Jac_Coif-Mey} we give a Coifman-Meyer type multiplier result for Jacobi trigonometric polynomials, here $L$ is the Jacobi operator.

The second main goal of this paper is to obtain fractional Leibniz rules for operators different from the  Laplacian. The fractional Leibniz rule states that, if $\Delta_{\RR^d}$ is the Laplacian on $\RR^d,$ then for each $s\geq 0$ and $1/p=1/p_1+1/p_2,$ $p_1,p_2>1,$ $p>1/2,$  we have,
$$\|(-\Delta_{\RR^d})^s (fg)\|_{p}\lesssim \|(-\Delta_{\RR^d})^s (f)\|_{p_1}\|g\|_{p_2}+\|(-\Delta_{\RR^d})^s (g)\|_{p_2}\|f\|_{p_1}.$$ The proof of this inequality can be found in Grafakos and Ou \cite{GraOh1}, see also Bourgain and Li \cite{BouLi1} for the endpoint case. The fractional Leibniz rule is also known as the Kato-Ponce inequality, as Kato and Ponce studied a similar estimate \cite{KatPon1} (see also \cite{KenPonVeg1}). Generalizations of Kato-Ponce or similar inequalities were considered by many authors. For example Muscalu, Pipher, Tao, and Thiele \cite{MusPipTaoThi1} extended this inequality by admitting partial fractional derivatives in $\RR^2$, Bernicot, Maldonado, Moen, and Naibo \cite{BerMalMoeNai1} proved the
Kato-Ponce inequality in weighted Lebesgue spaces, while Frey \cite{Fre1} obtained a fractional Leibniz rule for general operators satisfying Davies-Gaffney estimates and $p_1=p=2,$ $p_2=\infty$

In the present paper we obtain fractional Leibniz rules of the form
$$\|L^s (fg)\|_{p}\lesssim_s \|L^s (f)\|_{p_1}\|g\|_{p_2}+\|L^s (g)\|_{p_2}\|f\|_{p_1},$$
where $s>0$ and $1/p_1+1/p_2=1/p,$ with $p_1,p_2,p>1,$ in two other settings. In Corollary \ref{cor:LeibDiscLap} we prove a fractional Leibniz rule for $L$ being the discrete Laplacian on $\ZZ^d,$ while in Corollary \ref{cor:DunklLap_Leib} we justify a fractional Leibniz rule when $L$ is the Dunkl Laplacian in the product setting. The proofs of these fractional Leibniz rules rely on two properties of $L.$ Firstly, we need  appropriate Coifman-Meyer type multiplier results; these are Theorems \ref{thm:discLap_Coif-Mey} and \ref{thm:DunklLap_Coif-Mey} and are deduced from Theorem \ref{thm:genCoif-Mey}. Secondly, we require the existence of certain operators related to $L$ that satisfy (or almost satisfy) an integer order Leibniz rule. As we do not know such an operator  in the Jacobi setting we do not provide a fractional Leibniz rule there.

The article is organized as follows. In Section \ref{sec:Gen} we provide a general Coifman-Meyer type multiplier result, see Theorem \ref{thm:genCoif-Mey}. This is then a basis to establish Coifman-Meyer type multiplier results in various cases. In Section \ref{sec:Zd} we apply Theorem \ref{thm:genCoif-Mey} for the discrete Laplacian on $\ZZ^d,$ see Theorem \ref{thm:discLap_Coif-Mey}. As a consequence, in Corollary \ref{cor:LeibDiscLap} we also obtain a fractional Leibniz rule. Next, in Section \ref{sec:Dun} we deduce from Theorem \ref{thm:genCoif-Mey} a Coifman-Meyer multiplier theorem for general bi-radial Dunkl multipliers, see Theorem \ref{thm:DunklLap_Coif-Mey}. From this result we obtain a fractional Leibniz rule for the Dunkl Laplacian in the product case, see Corollary \ref{cor:DunklLap_Leib}. Finally, in Section \ref{sec:Jac}, using Theorem \ref{thm:genCoif-Mey} we prove a bilinear multiplier theorem for Jacobi trigonometric polynomial expansions.

It is straightforward to extend the result of this paper to the multilinear setting. However, to keep the presentation simple, we decided to limit ourselves to the bilinear case.

Troughout the paper we use the variable constant convention, where $C,$ $C_p,$ $C_s,$ etc. may denote different constants that may change even in the same chain of inequalities. We write $X\lesssim Y,$ whenever $X\leq C Y,$ with $C$ being independent of significant quantities.  Similarly, by $X\approx Y$ we mean that $C^{-1} Y\leq X \le C Y.$ By $\mS(\RR^d)$ we denote the space of Schwartz functions. The symbols $\ZZ$ and $\NN$ denote the sets of integers and non-negative integers, respectively.  For a multi-index $\alpha\in \NN^2$ by $|\alpha|$ we denote its length $\alpha_1+\alpha_2.$ Throughout the paper, for a function $\psi\colon [0,\infty)\to \mathbb{C}$ we set $$\psi_k(\la)=\psi(2^{-k}\la),\qquad \la \in [0,\infty).$$

 \section{General bilinear multipliers}
 
 \label{sec:Gen}

We say that a function $\mu \colon(0,\infty)\to \mathbb{C}$ satisfies the (one-dimensional) Mikhlin-H\"ormander condition of order $\rho\in \NN$ if it is differentiable up to order $\rho$ and 
\begin{equation}
\label{eq:Mh}
\|\mu\|_{MH(\rho)}:=\sup_{j\leq \rho}\sup_{\la\in (0,\infty)}|\la^j||\frac{d^j}{d\la^j} \mu(\la)|<\infty.
\end{equation}
Similarly, we say that $m\colon (0,\infty)^2\to \mathbb{C}$ satisfies the (two-dimensional) Mikhlin-H\"ormander condition of order $s\in \NN,$ if the partial derivatives $\partial^{\alpha} m$ exist for multi-indices  $|\alpha|\leq s$  and 
\begin{equation}
 \label{eq:Hormcond}
\|m\|_{MH(s)}:=\sup_{|\alpha|\leq s}\,\sup_{\la\in (0,\infty)^2}|\la|^{|\alpha|}|\partial^{\alpha} m(\la_1,\la_2)|<\infty.
\end{equation}

Consider a non-negative self-adjoint operator $L$ on $L^2(X,\nu)$ with domain $\Dom(L).$  Here $(X,\nu)$ is a $\sigma$-finite measure space  with $\nu$ being a Borel measure. Throughout the paper we assume that $L$ generates a symmetric contraction semigroup, namely
\label{p:CT} \begin{equation}
\tag{CT}
 \label{eq:Con}
\|e^{-tL}f\|_{L^p(X,\nu)}\leq \|f\|_{L^p(X,\nu)},\qquad f\in L^p(X,\nu)\cap L^2(X,\nu),\end{equation}
and that $L$ is injective on $\Dom (L).$ Then, for $\mu\colon (0,\infty)\to \mathbb{C},$ the spectral theorem allows us to define the multiplier operator
$\mu(L)=\int_{(0,\infty)}\mu(\la)dE(\la)$
on the domain
$$\Dom(\mu(L))=\bigg\{f\in L^2(X,\nu)\colon \int_{(0,\infty)}|\mu(\la)|^2\,dE_{f,f}(\la)<\infty\bigg\}.$$
Here $E$ is the spectral measure of $L,$ while $E_{f,f}$ is the complex measure defined by $E_{f,f}(\cdot)=\langle E(\cdot)f,f\rangle_{L^2(X,\nu)}.$ 

We shall need the following assumption on $L$;
\label{p:asum}
\begin{equation}
\tag{MH}   \label{eq:MH} \parbox{\dimexpr\linewidth-5em}{$L$ has a Mikhlin-Hörmander functional calculus of a finite order $\rho>0.$ More precisely, every function $\mu$ that satisfies \eqref{eq:Mh} gives rise to an operator $\mu(L)$ which is bounded on all $L^p(X,\nu),$ $1<p<\infty,$ and $$\|\mu(L)\|_{L^p(X,\nu)\to L^p(X,\nu)}\leq C_p \|\mu\|_{MH(\rho)}.$$}
\end{equation}
Note that if $L=(-\Delta_{\RR})^{1/2}$ then \eqref{eq:MH} follows from the  Mikhlin-H\"ormander multiplier theorem.

There are two consequence of  \eqref{eq:MH} which will be needed later. The first of them is well known and follows from Khintchine's inequality.
\begin{pro}
	\label{thm:CkimpliesSquare}
	Let $\psi\colon [0,\infty)\to \mathbb{C}$ be a function supported in $[\varepsilon,\varepsilon^{-1}],$ for some $\varepsilon>0$, and assume that $\psi\in C^{\rho}([0,\infty))$. Then the square function $$f\mapsto S_{\psi}(f)=\big(\sum_{k\in \mathbb{Z}}|\psi_k( L)f|^2\big)^{1/2}$$ is bounded on $L^p(X,\nu),$ $p>1,$ and
	\begin{equation}
	\label{eq:Squarebound}
	\|S_{\psi}(f)\|_{L^p(X,\nu)}\leq C_{\varepsilon}\, \|\psi\|_{C^{\rho}([0,\infty))}\|f\|_{L^p(X,\nu)}.
	\end{equation}
\end{pro}

The second of the required consequences is proved in \cite[Corollary 3.2]{WrocMHc}.
\begin{pro}
	\label{thm:CkimpliesMaximal}
	Let $\varphi\colon [0,\infty)\to \mathbb{C}$ be compactly supported, and assume that $\varphi\in C^{\alpha}([0,\infty))$ for some $\alpha>\rho+2.$ Then the maximal operator $$f\mapsto M_{\varphi}(f)=\sup_{k\in \mathbb{Z}}|\varphi_k( L)f|$$ is bounded on $L^p(X,\nu),$ $p>1,$ and
	\begin{equation}
	\label{eq:Maxbound}
	\|M_{\varphi}(f)\|_{L^p(X,\nu)}\leq \|\varphi\|_{C^{\rho+2}([0,\infty))}\|f\|_{L^p(X,\nu)}.
	\end{equation}
\end{pro}

 To simplify the proof of our main Theorem \ref{thm:genCoif-Mey} we will need an auxiliary subspace of $L^2(X,\nu).$ Namely, consider the spaces
 \begin{equation}
\label{eq:Aspaces}
\mA_2=\{g\in L^2(X,\nu)\colon g=E_{(\varepsilon,\varepsilon^{-1})}g,\,\textrm{for some $\varepsilon>0$}\}\, \textrm{ and } \mA=\mA_2\cap \bigcap_{1<p<\infty} L^p(X,\nu). \end{equation} Then,  \eqref{eq:MH} implies that $\mA$ is dense in  $L^p(X,\nu)$ for $1<p<\infty.$ 

For the convenience of the reader we shall justify this statement. Let $\psi\colon [0,\infty)\to \mathbb{C}$ be a smooth function which is supported in $[1/2,2]$ and such that $\sum_{k\in \ZZ}\psi_k(\la)=1,$ $\la>0.$ Then, for each $N\in \NN$ and $f\in L^1(X,\nu)\cap L^{\infty}(X,\nu)$ the partial sum $S_N f=\sum_{k=-N}^N \psi_k(L)f$ belongs to $\mA$ by \eqref{eq:MH}. We claim that $S_N f\to f$ in $L^p(X,\nu).$ To see this we take $1<r<\infty$ if $p\le 2$ or $r>p$ if $p>2.$ Then we  observe that  $\|S_N f\|_{L^r(X,\nu)}$ is uniformly bounded in $N$ (this follows from \eqref{eq:MH}) and that $S_N f\to f$ in $L^2(X,\nu)$ (this follows from the spectral theorem, since $E_{\{0\}}=0$ by the injectivity of $L$). Therefore, the log-convexity of $L^p$ norms proves the claim. Finally, a density argument together with the fact that $\|S_N f\|_{L^p(X,\nu)}$ is uniformly bounded in $N$ shows that $\mA$ is dense in $L^p(X,\nu)$ and finishes our task.     
 
 Besides being dense in $L^p(X,\nu)$ the space $\mA$ has the nice property that each $f\in \mA$ satisfies $f=\sum_{k=-N(f)}^{N(f)} \psi_k (L)f,$ where $N(f)$ is a fixed integer depending on $f$ and $\psi$ is the function from the previous paragraph. This allows us to deal easily with some rather delicate questions on convergence in the proof of Theorem \ref{thm:genCoif-Mey}.

We proceed to define formally the bilinear multipliers studied in this paper. To do this we will need the operators $L_1=L \otimes I$ and $L_2=I\otimes L.$ These may be regarded as non-negative self-adjoint operators on $L^2(X\times X,\nu\otimes \nu),$ see \cite[Theorem 7.23]{schmu:dgen} and \cite[Proposition A.2.2]{PhD}. Moreover, the spectral measure of $L_1$ is $E_{L}\otimes I,$ while the spectral measure of  $L_2$ is $I\otimes E_{L}.$ Thus, the operators $L_1$ and $L_2$ commute strongly and the bivariate spectral theorem, see e.g. \cite[Theorem 5.21]{schmu:dgen}, allows us to consider multiplier operators $$m(L_1,L_2)=\int_{(0,\infty)^2}m (\la)\,dE^{\otimes}(\la)$$
on the domain
$$\Dom(m(L_1,L_2))=\bigg\{F\in L^2(X\times X,\nu\otimes \nu)\colon \int_{(0,\infty)}|m(\la)|^2\,dE^{\otimes}_{F,F}(\la)<\infty\bigg\}.$$
Here  $m\colon [0,\infty)^2\to \mathbb{C}$ is a Borel measurable function,   $E^{\otimes}=E_L\otimes E_L$ is the joint spectral measure of $(L_1,L_2),$ while $E^{\otimes}_{F,F}$ is the complex measure defined by $(E^{\otimes})_{F,F}(\cdot)=\langle E^{\otimes}(\cdot)F,F\rangle_{L^2(X\times X,\nu\otimes \nu)}.$

In the most general form the bilinear multiplier operators studied in the paper are given by
 \begin{equation}
 \label{eq:genmult}
 B_m(f_1,f_2)(x)=m(L_1,L_2)(f_1\otimes f_2)(x,x)=\left(\iint_{(0,\infty)^2}m(\la)\,d(E_{f_1}\otimes E_{f_2})\right)(x,x),
 \end{equation}
 where $L_1=L\otimes I$ and $L_2=I\otimes L.$ 
 Since the diagonal $\{(x,x)\colon x\in X\}$ may be of measure $0$ in $(X\times X,\nu\otimes \nu)$ the equation \eqref{eq:genmult} is not formal. In order to make it rigorous we assume that:
 \label{p:WD}
 \begin{equation}
 \label{eq:WD} \tag{WD}
 \parbox{\dimexpr\linewidth-5em}{
  if $f_1,f_2\in \mA$  and $m\colon(0,\infty)^2\to \mathbb{C}$ is bounded, then \begin{itemize}
  \item  $m(L_1,L_2)(f_1\otimes f_2)$ has a continuous representative on $X\times X$
\item $\|m(L_1,L_2)(f_1\otimes f_2)\|_{L^{\infty}((X\times X),\nu\otimes \nu)}\leq C_{f_1,f_2}\|m\|_{L^{\infty}((0,\infty)^2)}.$
  \end{itemize} 
}
 \end{equation} 
Thus, restricting $m(L_1,L_2)(f_1\otimes f_2)(x_1,x_2)$ to the diagonal, we have a formal definition of $B_m(f_1,f_2),$ for $f_1,f_2\in \mA.$ For instance, if $L=(-\Delta_{\RR})^{1/2},$ then the operator $B_m$ is closely related to the bilinear multiplier for the Fourier transform, namely,
$$ B_m(f_1,f_2)(x)=\iint_{\mathbb{R}^2}m(|\xi_1|,|\xi_2|)\hat{f}(\xi_1)\hat{f}(\xi_2)e^{ix(\xi_1+\xi_2)}\, d\xi.$$
If $m$ is bounded and $f_1,f_2\in \mA$ then $m(\xi_1,\xi_2)\hat{f}(\xi_1)\hat{f}(\xi_2)\in L^1(\RR^d),$ and thus $B_m(f_1,f_2)(x)$ is well defined (in fact continuous) by the Lebesgue dominated convergence theorem.

We need one more assumption to prove the main theorem. Namely, we require that:
\label{p:asumPF}
\begin{equation}
 \label{eq:PF} \tag{PF} 
  \parbox{\dimexpr\linewidth-5em}{
 there is $b>0$ with the following property: if $\varphi$ and $\psi$ are bounded smooth functions such that $\supp\varphi_k\subseteq [0,2^{k-b}]$ and  $\supp\psi_k\subseteq [2^{k-2},2^{k+2}],$ $k\in \ZZ,$ then
	$$\varphi_k(L)(f_1)\cdot \psi_k(L)(f_2)=\tpsi_k(L)[\varphi_k(L)(f_1)\cdot \psi_k(L)(f_2)],\qquad \textrm{for } f_1,f_2 \in \mA,$$
	where $\tpsi_k$ is a smooth function which is bounded by $1,$ equals $1$ on $[2^{k-3-b},2^{k+3+b}]$ and vanishes outside $[2^{k-5-b},2^{k+5+b}].$}
\end{equation} 
We remark that, since $f_1,f_2\in \mA,$ the function  $g=\varphi_k(L)(f_1)\cdot \psi_k(L)(f_2)$ belongs to $ L^2(X,\nu),$ so that an application of $\tpsi_k(L)$ to $g$ is legitimate. Note that when $L=(-\Delta_{\RR})^{1/2}$ the formula \eqref{eq:PF} can be easily deduced by using the convolution structure on the frequency space associated with Fourier multipliers. 

In what follows we often abbreviate $L^p:=L^p(X,\nu)$ and $\|\cdot\|_p:=\|\cdot\|_{L^p}.$ Let $p,p_1,p_2>1.$ We say that a bilinear operator $B$ is bounded from $L^{p_1}\times L^{p_2}$ to $L^p$ if
$$\|B(f_1,f_2)\|_p\leq C \|f_1\|_{p_1}\|f_2\|_{p_2},\qquad f_1,f_2\in \mA.$$ Note that in this case $B$ has a unique bounded extension from $L^{p_1}\times L^{p_2}$ to $L^p.$


The main result of this paper is a Coifman-Meyer type general bilinear multiplier theorem. 
\begin{thm}
\label{thm:genCoif-Mey} Let $L$ be a non-negative self-adjoint operator on $L^2(X,\nu),$ which is injective on its domain and satisfies \eqref{eq:Con}, \eqref{eq:MH}, \eqref{eq:WD}, and \eqref{eq:PF}. Assume that $m\colon (0,\infty)^2\to \mathbb{C}$ satisfies the Mikhlin-H\"ormander condition \eqref{eq:Hormcond} of an order $s>2\rho+4$. Then the bilinear multiplier operator $B_m,$ given by \eqref{eq:genmult}, is bounded from $L^{p_1}\times L^{p_2}$ to $L^p,$ where $1/p_1+1/p_2=1/p,$ and $p_1,p_2,p>1.$ Moreover, for such $p,p_1,p_2,$ there is $C=C(p_1,p_2,p,s)$ such that
\begin{equation}
\label{eq:thm_gen_bound}
\|B_m(f_1,f_2)\|_{p}\leq C\, \|m\|_{MH(s)}\,\|f_1\|_{p_1}\|f_2\|_{p_2}.
\end{equation}
\end{thm}
\begin{proof}
	Let $\psi$ be a smooth function supported in $[1/2,2]$ and such that $\sum_k \psi_k \equiv 1.$ We set $F=f_1\otimes f_2\colon X\times X \to \mathbb{C}$ and split
\begin{equation*}
	\begin{split}
&B_m(f_1,f_2)(x)=\sum_{k_1,k_2\in \mathbb{Z}} [\psi_{k_1}(L_1)\psi_{k_2}(L_2)m(L_1,L_2)](F)(x,x)\\
&=\sum_{|k_1-k_2|\leq b+2}\ldots + \sum_{k_1>k_2+b+2}\ldots + \sum_{k_2>k_1+b+2}\ldots:=T_1+T_2+T_3.
\end{split}
\end{equation*}
There is no issue of convergence here as for $f_1,f_2\in \mA$ each of the sums defining $T_1,T_2,$ and $T_3$ is finite. 

We estimate separately each of the operators $T_i,$ $i=1,2,3,$ starting with $T_1.$ This is the easiest part, in fact here the assumption \eqref{eq:PF} is redundant.

For $k\in \ZZ$ set $$m_k(\la_1,\la_2)=\psi_{k}(\la_1)\sum_{k_2\colon|k-k_2|\leq b+2}\psi_{k_2}(\la_2)m(\la)=\psi_{k}(\la_1)\phi_k(\la_2)m(\la),$$
with $\phi(\la_2)=\sum_{|j|\leq b+2}\psi_j(\la_2),$ so that $\supp \phi \subseteq [2^{-b-3},2^{b+3}],$ and $$\supp \psi \otimes \phi \subseteq [2^{-1},2^1]\times  [2^{-b-3},2^{b+3}].$$  Let $\tpsi$ be another smooth function, which vanishes outside $[2^{-b-4},2^{b+4}]$ and equals $1$ on $[2^{-b-3},2^{b+3}].$  Then
$$m_k(\la_1,\la_2)=[\tpsi_k(\la_1)\tpsi_k(\la_2)]\psi_{k}(\la_1)\phi_k(\la_2)m(\la),$$
Moreover, $\supp m_k\subseteq [2^{k-b-4},2^{k+b+4}]^2,$ and, consequently, $M_k(\la):=m_k(2^{k}\la)$ is supported in $[-2^{b+4},2^{b+4}]^2:=[-a,a]^2.$ Thus, $M_k$ can be expanded into a double Fourier series inside $[-a,a]^2,$ i.e.,
$$M_k(\la)=\sum_{n_1,n_2\in \mathbb{Z}}c_{n,k}e^{\pi i n_1 \la_1/a}e^{\pi i n_2 \la_2/a},\qquad \la \in [-a,a]^2,$$
with the Fourier coefficients
$$c_{n,k}=\frac1{4a^2}\iint_{[-a,a]^2}[\psi\otimes \phi] m(2^k\xi)\, e^{\pi i n_1 \xi_1/a}e^{\pi i n_2 \xi_2/a}\,d\la.$$
Now, using integration by parts, together with the assumption \eqref{eq:Hormcond}, and the fact that $\psi\otimes \phi$ is compactly supported away from $0$, we we obtain the uniform in $k\in \ZZ$ bound
\begin{equation}
\label{eq:decaycnk}
|c_{n,k}|\leq C\,\|m\|_{MH(s)}\, (1+|n|)^{-s}, \qquad n\in\mathbb{Z}^2.
\end{equation}
We remark that here, in order to conclude \eqref{eq:decaycnk}, it is perfectly enough to assume the Marcinkiewicz 'product' condition $$|D^{\gamma}m(\la)|\leq C |\la_1|^{\gamma_1}|\la_2|^{\gamma_2},$$ instead of \eqref{eq:Hormcond}.

Coming back to $m_k$ we now write, for $\la \in [2^{k-b-4},2^{k+b+4}]^2,$
$$\psi_{k}(\la_1)\phi_k(\la_2)m(\la)=\sum_{n\in \mathbb{Z}^2} c_{n,k}\,  e^{2\pi i n_1 2^{-k}\la_1/a}e^{2\pi i n_2 2^{-k}\la_2/a}.$$
Thus, $m_k$ can be expressed as
\begin{align*}
&m_k(\la_1,\la_2)=\sum_{n\in \ZZ^2}c_{n,k} [\tpsi_k(\la_1)e^{(2\pi/a)  i n_1 2^{-k}\la_1}][\tpsi_k(\la_2)e^{(2\pi/a)  i n_2 2^{-k}\la_2}]\\
&:=\sum_{n\in \ZZ^2} c_{n,k}\psi_{k}^{n_1}(\la_1)\psi^{n_2}_{k}(\la_2).
\end{align*}
\label{p:L2pointeq}
By \eqref{eq:decaycnk} and the bivariate spectral theorem we have that
$$m_k(L_1,L_2)(F)(x_1,x_2)=\sum_{n\in \ZZ^2}c_{n,k}\, [\tpsi_k(L_1)e^{(2\pi/a)  i n_1 2^{-k}L_1}(f_1)](x_1)[\tpsi_k(L_2)e^{(2\pi/a)  i n_2 2^{-k}L_2}](f_2)(x_2),$$
for a.e.\ $x_1,x_2\in X;$ here we have convergence in $L^2(X\times X,\nu \otimes \nu).$ Moreover, \eqref{eq:decaycnk} and the assumption \eqref{eq:WD} imply that the above sum converges also pointwise (and gives a continuous function on $X\times X$).

Consequently, for $x\in X$ we have
\begin{align*}
T_1(f_1,f_2)(x)= \sum_{k\in \ZZ} m_k(L_1,L_2)(F)(x,x)=\sum_{n\in \ZZ^2}\, \sum_{k\in \ZZ} c_{n,k}\, \psi_{k}^{n_1}(L)(f_1)(x)\cdot \psi_{k}^{n_2}(L)(f_2)(x),
\end{align*}
where we have used the fact that the sum in $k$ is finite when $f_1,f_2\in \mA.$ 
Now Schwarz's inequality (first inequality below), and H\"older's inequality together with \eqref{eq:decaycnk} (second inequality below), lead to the estimate
\begin{equation}
\label{eq:estS1}
\begin{split}
&\|T_1(f_1,f_2)\|_{p}\leq \sum_{n\in \ZZ^2}\, \sup_{k\in \ZZ}|c_{n,k}|\left\|\big(\sum_{k\in \ZZ} |\psi_{k}^{n_1}(L)(f_1)|^2\big)^{1/2}\big(\sum_{k\in \ZZ} |\psi_{k}^{n_2}(L)(f_2)|^2\big)^{1/2}\right\|_{p}\\
&\lesssim  \|m\|_{MH(s)}\,\sum_{n\in \ZZ^2} (1+|n|)^{-s} \left\|\big(\sum_{k\in \ZZ} |\psi_{k}^{n_1}(L)(f_1)|^2\big)^{1/2}\right\|_{p_1}\left\|\big(\sum_{k\in \ZZ} |\psi_{k}^{n_2}(L)(f_2)|^2\big)^{1/2}\right\|_{p_2}.
\end{split}
\end{equation}
 Thus, taking into account the presence of the modulations $e^{2\pi i n_j 2^{-k}\la_j/a}$ in the definition of $\psi_{k}^{n_j},$ $j=1,2,$ and using Proposition \ref{thm:CkimpliesSquare} we obtain
 $$ \left\|\bigg(\sum_{k\in \ZZ} |\psi_{k}^{n_j}(L)(f_j)|^2\bigg)^{1/2}\right\|_{p_j}\lesssim (1+|n_j|)^{\rho}\,\|f_j\|_{p_j}.$$
 However, since we have the rapidly decaying factor in \eqref{eq:estS1}, if $s>2\rho+4,$ we arrive at the desired bound
 $$\|T_1(f_1,f_2)\|_{p}\lesssim \|m\|_{MH(s)}\,\|f_1\|_{p_1}\|f_2\|_{p_2}.$$

 Now we pass to estimating $T_2$ and $T_3.$ Since the proofs are mutatis mutandis the same, we treat only the former operator. Setting $\varphi=\sum_{j<-b-2}\psi_j$ we rewrite $T_1$ as
 \begin{align*}
 T_2(f_1,f_2)(x)&=\sum_{k_1>k_2+b+2}[\psi_{k_1}(L_1)\psi_{k_2}(L_2)m(L_1,L_2)](F)(x,x)\\
 &=\sum_{k_1} [\psi_{k_1}(L_1)\bigg(\sum_{k_2<k_1-b-2} \psi(2^{-k_2}L_2)\bigg)m(L_1,L_2)](f_1\otimes f_2)(x,x)\\
 &= \sum_k [\psi_{k}(L_1)\varphi_k(L_2)m(L_1,L_2)](f_1\otimes f_2)(x,x),
 \end{align*}
where $\varphi(\la_2)=\sum_{k_2<-b-2}\psi_{k_2}(\la_2).$ Then clearly $\supp \varphi \subseteq [0,2^{-b-1}].$ Recall that in the above decomposition of $T_2$ all the appearing sums in $k,$ $k_1,$ and $k_2,$ are in fact finite since $f_1,f_2\in \mA.$ 

Set $m_k:=\psi_{k}\varphi_k m$ and note that $m_k$ is supported in $[2^{k-1},2^{k+1}]\times [0,2^{k-b-1}],$ this is because $$\supp \psi \otimes \varphi \subseteq [2^{-1},2^{1}]\times [0,2^{-b-1}].$$
 Similarly to the case of $T_1$ we expand the function $M_k=m_k(2^{k}\la)$ in a Fourier series. Namely, let $\tpsi$ be a smooth function vanishing outside $[2^{-2},2^{2}]$ and equal to $1$ on $[2^{-1},2^{1}],$ and let $\tfi$ be a smooth function vanishing outside $[0,2^{-b}]$ and equal to $1$ on $[0,2^{-b-1}].$ Then
$$m_k(\la_1,\la_2)=[\tpsi_k(\la_1)\tfi_k(\la_2)]\psi_{k}(\la_1)\varphi_k(\la_2)m(\la),$$
Moreover, $\supp m_k\subseteq [2^{k-1},2^{k+1}]\times [0,2^{k-b-1}],$ and, consequently, $M_k(\la)=m_k(2^{k}\la)$ is supported in $[-2,2]^2.$ Hence, $M_k$ can be expanded into a double Fourier series inside $[-2,2]^2$, i.e., for $\la \in [-2,2]^2,$
$$M_k(\la)=\sum_{n_1,n_2\in \mathbb{Z}}c_{n,k}e^{\frac{\pi}{2} i n_1 \la_1}e^{\frac{\pi}{2} i n_2 \la_2},$$
with the Fourier coefficients
$$c_{n,k}=\frac1{16}\iint_{[-2,2]^2}[\psi\otimes \varphi] m(2^k\xi)\, e^{\frac{\pi}{2} i n_1 \xi_1}e^{\frac{\pi}{2} i n_2 \xi_2}\,d\xi.$$
As with $T_1$, we now use integration by parts, together with the assumption \eqref{eq:Hormcond}. Here it is important that we assume the stronger Mikhlin-Hörmander condition instead of merely the Mikhlin-Marcinkiewicz condition. Indeed, from integration by parts we obtain, for arbitrary $\beta$
$$c_{n,k}=O((1+|n|)^{-\beta})\iint_{[-2,2]^2}\frac{d^{\beta}}{d\xi^{\beta}}([\psi\otimes \varphi] m(2^k\xi))\, e^{\frac{\pi}{2} i n_1 \xi_1}e^{\frac{\pi}{2} i n_2 \xi_2}\,d\xi.$$
However, as $\psi\otimes \varphi$ does not vanish for $\la_2$ close to zero, in order to conclude that the above integral is uniformly bounded we do need \eqref{eq:Hormcond}. In summary we proved that \eqref{eq:decaycnk} holds also in this case.

Coming back to $m_k$ we now write, for $\la \in [2^{k-2},2^{k+2}]\times [0,2^{k-b}]$
$$\psi_{k}(\la_1)\varphi_k(\la_2)m(\la)=\sum_{n\in \mathbb{Z}^2} c_{n,k}\,  e^{\frac{\pi}{2} i n_1 2^{-k}\la_1}e^{\frac{\pi}{2} i n_2 2^{-k}\la_2}.$$
Thus, $m_k,$ $k\in \ZZ,$ can be expressed as
\begin{align*}
&m_k(\la_1,\la_2)=\sum_{n}c_{n,k} [\tpsi_k(\la_1)e^{\frac{\pi}{2} i n_1 2^{-k}\la_1}][\tfi_k(\la_2)e^{\frac{\pi}{2} i n_2 2^{-k}\la_2}]\\
&:=\sum_{n\in \ZZ} c_{n,k}\psi_{k}^{n_1}(\la_1)\varphi^{n_2}_{k}(\la_2).
\end{align*}
 With the aid of \eqref{eq:WD} and \eqref{eq:decaycnk}, arguing as on p. \pageref{p:L2pointeq} we see that
 $$m_k(L_1,L_2)(F)(x,x)=\sum_{n\in \ZZ^2}\, c_{n,k}\, \psi_{k}^{n_1}(L)(f_1)(x)\cdot\varphi_{k}^{n_2}(L)(f_2)(x),$$
where the series on the right converges pointwise to a continuous function on $X.$

Summarizing the above, we have just decomposed
\begin{align*}
T_2(f_1,f_2)(x)= \sum_{k\in \ZZ} m_k(L_1,L_2)(F)(x,x)=\sum_{n\in \ZZ^2}\, \sum_{k\in \ZZ} c_{n,k}\, \psi_{k}^{n_1}(L)(f_1)(x)\cdot\varphi_{k}^{n_2}(L)(f_2)(x).
\end{align*}
Now, let $\tpsi$ be a real-valued smooth function equal to $1$ on $[2^{-3-b},2^{3+b}]$ and vanishing outside $[2^{-5-b},2^{5+b}].$  Since, for each $n=(n_1,n_2)\in \ZZ^2,$ the function  $\varphi_{k}^{n_2}$ is supported in $[0,2^{k-b}],$ and the function $\psi_{k}^{n_1}$ is supported in $[2^{k-2},2^{k+2}],$ using the assumption  \eqref{eq:PF} we have
$$ T_2(f_1,f_2)(x)=\sum_{n\in \ZZ^2}\, \sum_{k\in \ZZ} c_{n,k} \tpsi_k(L)[\psi_{k}^{n_1}(L)(f_1)\cdot\varphi_{k}^{n_2}(L)(f_2)](x).$$ Hence, if $h$ is a function in $L^q,$ $1/p+1/q=1,$ then we obtain
\begin{equation*}
\int_X T_2(f_1,f_2)(x)h(x)\,d\nu(x)= \int_X \sum_{n\in \ZZ^2}\, \sum_{k\in \ZZ} c_{n,k} \,\psi_{k}^{n_1}(L)(f_1)(x)\cdot\varphi_{k}^{n_2}(L)(f_2)(x) \tpsi_k(L)(h)(x)\,d\nu(x),
\end{equation*}
and, consequently, 
\begin{equation}
\begin{split}
&\label{eq:estS2}|\int_X T_2(f_1,f_2)(x)h(x)\,d\nu|\leq \sum_{n\in\ZZ^2} \sup_{k\in \ZZ}|c_{n,k}|\times\\
& \times\int_X\bigg[\bigg(\sum_{k\in \ZZ} |\psi_{k}^{n_1}(L)(f_1)|^2\bigg)^{1/2}\sup_{k\in \ZZ} \varphi_{k}^{n_2}(L)(f_2)\bigg] \bigg(\sum_{k\in \ZZ} |\tpsi_{k}(L)(h)|^2\bigg)^{1/2}\,d\nu\\
&\lesssim \|m\|_{MH(s)}\, \sum_{n\in \ZZ^2} (1+|n|)^{-s} \left\|\big(\sum_{k\in \ZZ} |\psi_{k}^{n_1}(L)(f_1)|^2\big)^{1/2}\right\|_{p_1}\left\|\sup_{k\in \ZZ}| \varphi_{k}^{n_2}(L)(f_2)|\right\|_{p_2},
\end{split}
\end{equation}
where we used Proposition \ref{thm:CkimpliesSquare} with $\tpsi$ in the second inequality above. Similarly to the estimate for $T_1,$ applying Propositions \ref{thm:CkimpliesSquare} and \ref{thm:CkimpliesMaximal} leads to
 \begin{align*} &\left\|\bigg(\sum_{k} |\psi_{k}^{n_1}(L)(f_1)|^2\bigg)^{1/2}\right\|_{p_1}\lesssim (1+|n_1|)^{\rho}\|f_1\|_{p_1}\qquad\textrm{(cf.\ \eqref{eq:Squarebound})}\\
 &\left\|\sup_{k} |\varphi_{k}^{n_2}(L)(f_2)|\right\|_{p_2}\lesssim (1+|n_2|)^{\rho+2}\|f_2\|_{p_2}\qquad\textrm{(cf.\ \eqref{eq:Maxbound})}.
 \end{align*}
 Finally, the rapidly decaying factor in \eqref{eq:estS2} gives, for $s>2\rho+4$, the desired bound
 $$\|T_2(f_1,f_2)\|_{p}\lesssim \|m\|_{MH(s)}\,\|f_1\|_{p_1}\|f_2\|_{p_2}.$$
 The proof of Theorem \ref{thm:genCoif-Mey} is thus completed.
 \end{proof}

\section{Bilinear multipliers on $\mathbb{Z}^d$}
\label{sec:Zd}

In the present section we formalize Theorem \ref{thm:genCoif-Mey} for bilinear multiplier operators on $\mathbb{Z}^d.$ We also prove a fractional Leibniz rule for the discrete Laplacian. 

Let $e_j=(0,\ldots,1,\ldots,0)\in\mathbb{Z}^d$ be the $j$-th coordinate vector. Consider the discrete Laplacian on $\mathbb{Z}^d,$ given by $$\dLZd(f)(n)=2d\,f(n)-\sum_{j=1}^{d}(f(n+e_j)+f(n-e_j))=2dIf(n)-\sum_{j=1}^d (f*\delta_{e_j}+f*\delta_{-e_j}).$$ The multilinear operators \eqref{eq:genmult} for the discrete Laplacian are defined via Fourier analysis on $\mathbb{Z}^d.$ Namely, let $\TT^d\equiv (-1/2,1/2]^d$ be the d-dimensional torus, let  $$\mFZd(f)(\xi)=\sum_{n\in\mathbb{Z}^d}f(k)e^{2\pi i n\cdot \xi},\qquad \xi\in \TT^d$$ be the Fourier transform on $\mathbb{Z}^d,$ and define
$$\Sin^2(\xi)=4\sum_{j=1}^{d}\sin^2 (\pi \xi_j),\qquad \xi\in\TT^d.$$
 Then, since
$$\mFZd(\dLZd(f))(\xi)=\Sin^2(\xi) \mFZd(f)(\xi),\qquad \xi\in\TT^d,$$
the formula \eqref{eq:genmult} takes the form
\begin{equation}\label{eq:discLapMultDef}
\begin{split}
B_m(f_1,f_2)(n)&:=m((-\dLZd)^{1/2}\otimes I, I\otimes (-\dLZd)^{1/2})(f_1\otimes f_2)(n,n)\\ 
&=\int_{\TT^d}\int_{\TT^d} m(|\Sin(\xi_1)|,|\Sin(\xi_2)|)\mFZd(f_1)(\xi_1)\mFZd(f_2)(\xi_2)e^{-2\pi i n (\xi_1+\xi_2)}\,d\xi,
\end{split}
\end{equation}
where $n\in \ZZ^d.$ Note that the space $\mA_2$ from \eqref{eq:Aspaces} in this case is given by
$$\mA_2=\{g\in L^2(\TT^d)\colon \mFZd(g)(\xi)=0\textrm{ for some $\varepsilon>0$ and all $|\xi|<\varepsilon.$}\}$$ 

Throughout this section we denote by $L^p$ the space $l^p(\mathbb{Z}^d)$ equipped with the counting measure. Using Theorem \ref{thm:genCoif-Mey} we prove the following Coifman-Meyer multiplier theorem for the discrete Laplacian.

\begin{thm}
\label{thm:discLap_Coif-Mey} Assume that $m$ satisfies H\"ormander's condition \eqref{eq:Hormcond} of order $s>d+4.$ Then the bilinear multiplier operator given by \eqref{eq:discLapMultDef} is bounded from $L^{p_1}\times L^{p_2}$ to $L^p,$ where $1/p_1+1/p_2=1/p,$ and $p_1,p_2,p>1.$ Moreover, the bound \eqref{eq:thm_gen_bound} holds.
\end{thm}
\begin{proof}
It is well known that $L=(-\Delta_{\ZZ^d})^{1/2}$ is injective on $L^2$ and satisfies \eqref{eq:Con}. Moreover, it also satisfies \eqref{eq:WD} since for $f_1,f_2\in \mA$ we have $\mFZd(f_1)(\xi_1)\mFZd(f_2)(\xi_2)\in L^1(\TT^d\times \TT^d).$ From \cite[Theorem 1.1]{Alex_1} it follows that $-\dLZd$ has a Mikhlin-Hörmander functional calculus (of order $[d/2]+1$). Then, clearly, the same is true for $(-\dLZd)^{1/2}.$ Hence,  \eqref{eq:MH} has been justified.

To apply Theorem \ref{thm:genCoif-Mey} it remains to show that $L=(-\dLZd)^{1/2}$ satisfies \eqref{eq:PF}.  
We prove it with $b=7+\frac12 \log_2 d$. Since the spectrum of $(-\dLZd)^{1/2}$ is contained in $[0,2\sqrt d],$ we have $\psi_k((-\dLZd)^{1/2})\equiv 0,$ if $k>2+\frac12 \log_2 d.$ Hence, it suffices to show  \eqref{eq:PF} for $k\leq 2+\frac12 \log_2 d.$ Using elementary Fourier analysis on $\mathbb{Z}^d$ we see that to prove  \eqref{eq:PF} it is enough to show that $$\tpsi_k\circ |\Sin|=1\qquad\textrm{on the support of}\qquad ((\psi_k\circ |\Sin|)\mFZd(f_1)) *_{\mathbb{T}^d} ((\varphi_k\circ |\Sin|))\mFZd(f_2)),$$
where $\tpsi_k,$ $\psi_k,$ and $\varphi_k$ are the functions from \eqref{eq:PF}.  
In other words that we are left with proving that if $|\Sin(\xi)|<2^{k-3-b}$ or $|\Sin(\xi)|>2^{k+3+b},$ then
\begin{equation}
\label{eq:discIIasum}
\int_{\mathbb{T}^d}\psi_k(|\Sin(\xi-\eta)|)\mFZd(f_1)(\xi-\eta)\cdot   \varphi_k(|\Sin(\eta)|)\mFZd(f_2)(\eta)\,d\eta=0.
\end{equation}

The formula \begin{equation}\label{eq:trigform1}\sin \pi (t-s)=\sin\pi t \cos \pi s-\sin \pi s \cos \pi t,\qquad s,t \in \mathbb{T},\end{equation} leads to
$|\sin \pi (\xi_j)|\leq |\sin \pi (\xi_j-\eta_j)|+|\sin \pi \eta_j|,$ $j=1,\ldots,d,$ and, consequently,
$$|\Sin(\xi)|\leq \sqrt d (|\Sin(\xi-\eta)|+|\Sin(\eta)|), \qquad \eta\in \mathbb{T}^d.$$
From the above it follows that if $|\Sin(\xi)|>2^{k+3+b},$ then for every $\eta\in\mathbb{T}^d$ the integrand in \eqref{eq:discIIasum} vanishes.

It remains to show that also $ |\Sin(\xi)|<2^{k-3-b}$ forces \eqref{eq:discIIasum}. We argue by contradiction assuming that $ |\Sin(\xi)|<2^{k-3-b}$ yet the integral in \eqref{eq:discIIasum} is non-zero. Then, for some $\eta\in \mathbb{T}^d,$ we must have $\psi_k(|\Sin(\xi-\eta)|)\, \varphi_k(|\Sin(\eta)|)\neq 0,$ which implies that
\begin{equation}\label{eq:consCase2}
2^{k-1}\leq|\Sin(\xi-\eta)|\leq 2^{k+1}\qquad \textrm{and}\qquad |\Sin(\eta)|\leq 2^{k-b}.
\end{equation}
Note that since $k\le2+\frac 12 \log_2 d,$ the integral in \eqref{eq:discIIasum} runs over $|\Sin(\eta)|\leq 2^{k-b}\leq 2^{-1},$ and, consequently, we consider only those $\eta$ satisfying $|\cos \pi \eta_j|>\sqrt3/2>1/2,$ for every $j=1,\ldots,d.$ Now, using \eqref{eq:trigform1} (with $t-s=\xi_j,$ $s=-\eta_j$)  we obtain
\begin{align*}&|\sin \pi \xi_j|\geq |\sin \pi (\xi_j-\eta_j)||\cos\pi \eta_j|-|\cos\pi (\xi_j-\eta_j)||\sin\pi \eta_j|\\&\geq \frac 12|\sin \pi (\xi_j-\eta_j)|-|\sin\pi \eta_j|.\end{align*}
Summing the above estimate in $j$ and using Schwarz inequality we arrive at
\begin{align*}&\sqrt{d}|\Sin(\xi)|\geq \sum_{j=1}^d|\sin \pi \xi_j|\geq \frac12 \sum_{j=1}^d |\sin \pi (\xi_j-\eta_j)|-\sum_{j=1}^d |\sin\pi \eta_j|\geq \frac12 |\Sin(\xi-\eta)|-\sqrt{d}|\Sin(\eta)|.
\end{align*}
Now, since $ |\Sin(\xi)|<2^{k-3-b},$ using \eqref{eq:consCase2} we arrive at
$$2^{k-b-3}>|\Sin(\xi)|>\frac1{\sqrt{d}}\big(2^{k-1}-\sqrt{d}2^{k-b}\big)=\frac1{\sqrt{d}}(2^{k-1}-2^{k-7})>\frac1{\sqrt{d}}2^{k-2}=2^{k-b+5},$$
which is a contradiction.
\end{proof}

As a corollary of Theorem \ref{thm:discLap_Coif-Mey} we prove a fractional Leibniz rule for the discrete Laplacian on $\mathbb{Z}^d.$ For $\Real(z)\ge 0$ and $h\in L^2$ the complex derivative $(-\dLZd)^z h$ is given by
$$\mFZd((-\dLZd)^z h)(\xi)=|\Sin \xi|^{2z}\mFZd(h)(\xi),\qquad \xi \in \TT^d.$$ 
This coincides with taking the $n$-th composition of $(-\dLZd)$ when $z=n$ is a non-negative integer.  Clearly, $(-\dLZd)^z$ is bounded on $L^2.$ Moreover, when $z=s\in \RR,$ $s\ge 0,$ then $(-\dLZd)^s$ is also bounded on all $L^p,$ $1\leq p\leq \infty.$ To see this we just use the Taylor series expansion of the function $x^{s}=(1-(1-x))^{s},$ with $x$ replaced by $(-\dLZd)/(4d).$  This is legitimate since $(-\dLZd)/(4d)$ is a contraction on all $L^p$ spaces. Our fractional Leibniz rule is the following.
\begin{cor}
\label{cor:LeibDiscLap}
Let $1/p=1/p_1+1/p_2,$ with $p,p_1,p_2>1.$ Then, for every $s>0,$
\begin{equation}
\label{eq:LeibDiscLap}
\|(-\dLZd)^{s}(fg)\|_p\lesssim \|(-\dLZd)^{s}f\|_{p_1}\|g\|_{p_2}+ \|(-\dLZd)^{s}g\|_{p_2}\|f\|_{p_1},
\end{equation}
where $f,g\in \mA.$
\end{cor}
\begin{remark1}
	Note that if $f,g\in\mA$ then $fg\in L^2,$ hence $(-\dLZd)^{s}(fg)$ makes sense.  
\end{remark1}
\begin{remark2}
	Since $(-\dLZd)^{s}$ is bounded on all $L^p$ spaces, $1\leq p\leq \infty,$  a version of \eqref{eq:LeibDiscLap} without the Laplacians on the right hand side is obvious. This is in contrast with the fractional Leibniz rule on $\RR^d.$
\end{remark2}
In the proof of the corollary we shall need two lemmata. The first of them follows from the $l^p(\mathbb{Z})$ boundedness of the discrete Hilbert transform.
\begin{lem}
\label{lem:discLapSpecMult}
The one-dimensional linear multiplier operator
$$H(f)(n)=\int_0^{1/2}\mFZ(f)(x)e^{2\pi i \xi n}\,d\xi,\qquad n\in \mathbb{Z}$$
is bounded on all $l^p(\mathbb{Z})$ spaces, $1<p<\infty.$
\end{lem}
The second of the lemmata is the following.
\begin{lem}
	\label{lem:DiscFracPF}
	Let $d=1.$ Assume that $\varphi\colon (0,\infty)^2\to \bC$ is a bounded function that satisfies the Mikhlin-H\"ormander condition \eqref{eq:MH} of order $6.$ Then, for $\Real(z)\geq 0$ we have 
	\begin{equation}
	\label{eq:Claim_DiscLap_Leib}
	\begin{split}
	&(-\dLZ)^{z}(B_{\varphi}(f, g))(n)\\
	&=\iint_{\TT^2} \varphi(2|\sin \pi\xi_1|,2|\sin \pi\xi_2|)\,|2\sin\pi(\xi_1+\xi_2)|^{2z} e^{2\pi i (\xi_1+\xi_2)n}\, \mFZ(f)(\xi_1)\mFZ(g)(\xi_2)\,d\xi,
	\end{split}
\end{equation}	where $f,g\in \mA,$ and  $n\in \ZZ.$
\end{lem}
\begin{proof}
From Theorem \ref{thm:discLap_Coif-Mey} and the assumptions on $\varphi$ it follows that $B_{\varphi}(f,g)\in \ell^2(\ZZ).$ Thus, the left hand side of \eqref{eq:Claim_DiscLap_Leib} makes sense as a function on $\ell^2(\ZZ).$ Moreover, a continuity argument shows that it suffices to demonstrate  \eqref{eq:Claim_DiscLap_Leib} for $\Real(z)>0.$ 

	
	Set $\tilde{\varphi}(\xi_1,\xi_2)= \varphi(2|\sin \pi\xi_1|,2|\sin \pi\xi_2|).$ 
	 Since $-\dLZ(e^{2\pi it \cdot})(n)=4(\sin^2 \pi t) e^{2\pi it n},$ for $t\in \TT$ and $n\in \mathbb{Z},$
	 we deduce that $(-\dLZ)^{k}(e^{2\pi it \cdot})(n)=2^{2k}|\sin \pi t|^{2k} e^{2\pi it n},$ $k\in \mathbb{N}.$
	 Hence, for $k,n\in \mathbb{N},$ we have
	 \begin{equation*}
	 (-\dLZ)^{k}(B_{\varphi}(f, g))(n)=\iint_{\TT^2} \tilde{\varphi}(\xi_1,\xi_2)\,(4\sin^2\pi(\xi_1+\xi_2))^k e^{2\pi i (\xi_1+\xi_2)n}\, \mFZ(f)(\xi_1)\mFZ(g)(\xi_2)\,d\xi.
	 \end{equation*}
	 Thus, for $P$ being a polynomial we obtain
	 $$P(-\dLZ)(B_{\varphi}(f, g))(n)=\iint_{\TT^2} \tilde{\varphi}(\xi_1,\xi_2)\, P(4\sin^2\pi(\xi_1+\xi_2)) e^{2\pi i (\xi_1+\xi_2)n}\, \mFZ(f)(\xi_1)\mFZ(g)(\xi_2)\,d\xi,$$
	 where $n\in \ZZ.$
	 
	 Finally, a density argument shows that the above formula remains true for continuous functions in place of polynomials.  In particular, taking $\la\mapsto \la^{z},$ $\Real(z)>0,$ we obtain \eqref{eq:Claim_DiscLap_Leib}. 
\end{proof}

We proceed to the proof of the corollary.
\begin{proof}[Proof of Corollary \ref{cor:LeibDiscLap}]
We claim that it is enough to prove the corollary in dimension $d=1.$ Indeed, fix $s>0$ and assume that \eqref{eq:LeibDiscLap} is true in this case. Let $\dLZ$ be the one dimensional discrete Laplacian on $\mathbb{Z}.$ Define $L_j:=-\dLZ\otimes I_{(j)},$ $j=1,\ldots,d,$ to be the one-dimensional discrete Laplacian acting on the $j$-th variable, so that, clearly, $-\dLZd=\sum_{j=1}^d L_j.$ Since each $L_j$ generates a symmetric contraction semigroup, using e.g.\ the multivariate multiplier theorem \cite[Corollary 3.2]{WroJSMMSO}
we see that the operator $$(\sum L_j)^{s}(\sum L_j^{s})^{-1}$$ is bounded on $L^p,$ $p>1.$ In other words, we have the bound
\begin{equation*}\|(-\dLZd)^{s}(fg)\|_p\lesssim \|\sum_{j=1}^{d} L_j^{s}(fg)\|_{p}\leq \sum_{j=1}^{d}\|L_j^{s}(fg)\|_{p}.\end{equation*}
Since the multiplier $L_j^s (-\dLZd)^{-s}$ is bounded on all $L^p,$ $p>1,$ (this again follows from \cite[Corollary 3.2]{WroJSMMSO}) in order to conclude the proof of our claim it is thus enough to show that
\begin{equation}
\label{eq:LeibDiscLapAux1}
\|L_j^{s}(fg)\|_{p}\lesssim \|L_j^{s}f\|_{p_1}\|g\|_{p_2}+ \|L_j^{s}g\|_{p_2}\|f\|_{p_1},
\end{equation}
for every $j=1,\ldots,d.$

For notational simplicity we justify \eqref{eq:LeibDiscLapAux1} only for $j=1,$ the proofs for other $j$ are analogous. For a sequence $h\colon \mathbb{Z}^d\to \mathbb{C}$ denote $h_{n}(k):=h(k,n),$ $k\in\mathbb{Z},$ $n\in \mathbb{Z}^{d-1}.$ Clearly, we have $(fg)_n(\cdot)=f_n(\cdot)g_n(\cdot).$ Then, using \eqref{eq:LeibDiscLap} in the dimension $d=1$ (first inequality below), together with the simple fact that $(a+b)^{p}\approx a^{p}+b^p$ (second and last inequalities below), and  H\"older's inequality with exponents $p_1/p, p_2/p>1$ (third inequality below) we obtain
\begin{align*}\|L_1^{s}(fg)\|_{p}&=\sum_{n\in \mathbb{Z}^{d-1}} \|L_1^{s}((fg)_n(\cdot))\|_{l^p(\mathbb{Z})}^{p}=\sum_{n\in \mathbb{Z}^{d-1}} \|L_1^{s}(f_n(\cdot)g_n(\cdot))\|_{l^p(\mathbb{Z})}^{p}
\\&\lesssim \sum_{n\in \mathbb{Z}^{d-1}} \big(\|L_1^{s}(f_n)\|_{l^{p_1}(\mathbb{Z})}\|g_n\|_{l^{p_2}(\mathbb{Z})}+\|L_1^{s}(g_n)\|_{l^{p_1}(\mathbb{Z})}\|f_n\|_{l^{p_2}(\mathbb{Z})}\big)^{p}\\
&\lesssim \sum_{n\in \mathbb{Z}^{d-1}} \|L_1^{s}(f_n)\|_{l^{p_1}(\mathbb{Z})}^{p}\|g_n\|_{l^{p_2}(\mathbb{Z})}^{p}+\|L_1^{s}(g_n)\|_{l^{p_1}(\mathbb{Z})}^{p}\|f_n\|_{l^{p_2}(\mathbb{Z})}^{p}\\
&\lesssim
\big(\sum_{n\in \mathbb{Z}^{d-1}} \|L_1^{s}(f_n)\|_{l^{p_1}(\mathbb{Z})}^{p_1}\big)^{p/p_1}\big(\sum_{n\in \mathbb{Z}^{d-1}}\|g_n\|_{l^{p_2}(\mathbb{Z})}^{p_2}\big)^{p/p_2}\\
&+\big(\sum_{n\in \mathbb{Z}^{d-1}} \|L_1^{s}(g_n)\|_{l^{p_2}(\mathbb{Z})}^{p_2}\big)^{p/p_2}\big(\sum_{n\in \mathbb{Z}^{d-1}}\|f_n\|_{l^{p_1}(\mathbb{Z})}^{p_1}\big)^{p/p_1}=\|L_1^s(f)\|_{p_1}^p\|g\|_{p_2}^p+\|L_1^s(g)\|_{p_2}^p\|f\|_{p_1}^p\\
&\lesssim \big(\|L_1^s(f)\|_{p_1}\|g\|_{p_2}+\|L_1^s(g)\|_{p_2}\|f\|_{p_1}\big)^{p}.
\end{align*}
Hence, \eqref{eq:LeibDiscLapAux1} is proved.

Having justified the claim we now focus on proving \eqref{eq:LeibDiscLap} for $d=1.$ Till the end of the proof of the corollary we work on $\mathbb{Z}$ and write $l^p$ and $\|\cdot\|_p$ for $l^p(\mathbb{Z})$ and $\|\cdot\|_{l^p(\mathbb{Z})},$ respectively. 

Let $\eta_0$ and $\eta_{1}$ be smooth functions satisfying $\supp \eta_0 \subseteq [0,1/4],$ $\supp \eta_{1}\subseteq [1/8,10]$ and $\eta_0+\eta_{1}=1$ on $[0,4].$ For a function $h\in \mA$ we set $h_0=\eta_0((-\dLZ)^{1/2})(h)$ and $h_{1}= \eta_{1}((-\dLZ)^{1/2})(h),$ so that $h=h_0+h_{1}.$ From \cite[Theorem 1.1]{Alex_1} it follows that, for each fixed $s>0$ the multiplier $(-\dLZ)^{-s}\eta_{1}(-\dLZ)$ is bounded on all $l^p,$ $1<p<\infty.$ Moreover, $h_0,h_1\in \mA.$ Since $h_{1}=(-\dLZ)^{-s}\eta_{1}(-\dLZ)[(-\dLZ)^{s}(h)]$, we thus have the estimate
$$\|h_{1}\|_{p}\lesssim \|(-\dLZ)^{s}h\|_{p}.$$
Hence, using the boundedness of $(-\dLZ)^{s}$ and H\"older's inequality we obtain
\begin{align*}
\|(-\dLZ)^{s}(f_{i_1}g_{i_2})\|_{p}\lesssim \|f_{i_1}\|_{p_1}\|g_{i_2}\|_{p_2}\lesssim \|(-\dLZ)^{s}f\|_{p_1}\|g\|_{p_2}+ \|(-\dLZ)^{s}g\|_{p_1}\|f\|_{p_2},
\end{align*}
for $i_1, i_2\in \{0,1\}$ not both equal to $0.$ In summary, to finish the proof it is enough to demonstrate that
\begin{equation*}
\|(-\dLZ)^{s}(f_{0}g_{0})\|_{p}\lesssim \|(-\dLZ)^{s}f\|_{p_1}\|g\|_{p_2}+ \|(-\dLZ)^{s}g\|_{p_1}\|f\|_{p_2}.
\end{equation*}

 Clearly, $\mFZ(f_0)(x)=\eta_0(|\sin \pi x|)\mFZ(f)(x)$ and $\mFZ(g)(y)=\eta_0(|\sin \pi y|)\mFZ(g_0)(y).$ Hence, denoting $I=[0,1/2)$ and using Lemma \ref{lem:DiscFracPF} together with \eqref{eq:trigform1} we now write
\begin{align*}
&(-\dLZ)^s(f_0g_0)(n)\\
&=2^{2s}\int_\TT \int_\TT [|\sin\pi \xi_1\cos\pi \xi_2+\sin \pi \xi_1\cos \pi \xi_2|^{2s}\eta_0(|\sin \pi \xi_1|)\eta_0(|\sin \pi \xi_2|)]\\ &\hspace{1.6cm}\times e^{2\pi i (\xi_1+\xi_2)n}\,\mFZ(f)(\xi_1)\mFZ(g)(\xi_2)\,d\xi\\
&=2^{2s}\sum_{\epsilon\in\{-1,1\}^2}
\int_{\epsilon_1 I}\int_{\epsilon_2 I}|\sin\pi \xi_1\sqrt{1-\sin^2 \pi \xi_2}+\sin\pi \xi_2 \sqrt{1-\sin^2 \pi \xi_1}|^{2s}\eta_0(|\sin \pi \xi_1|)\eta_0(|\sin \pi \xi_2|)\\
&\hspace{1.6cm}\times e^{2\pi i (\xi_1+\xi_2)n}\,\mFZ(f)(\xi_1)\mFZ(g)(\xi_2)\,d\xi :=\sum_{\epsilon\in\{-1,1\}^2} T_{\epsilon}(f,g)(n),\qquad n\in \ZZ.
\end{align*}
Thus, in order to finish the proof it is enough to show that, for $\epsilon\in\{-1,1\}^2$ it holds
\begin{equation}
\label{eq:discLapTjest}
\|T_{\epsilon}(f,g)\|_{p}\lesssim \|(-\dLZ)^{s}f\|_{p_1}\|g\|_{p_2}+ \|(-\dLZ)^{s}g\|_{p_2}\|f\|_{p_1}.
\end{equation}
It is enough to justify \eqref{eq:discLapTjest} only for $T_{1,1}$ and $T_{1,-1}$ as the proofs for $T_{-1,1}$ and $T_{-1,-1}$ are symmetric.
In what follows we let $\phi$ be a function in $C^{\infty}([0,\infty))$ supported in $[0,1/4]$ and such that $\phi(t)+\phi(t^{-1})=1.$ Note that then $\phi(\la_2/\la_1)$ satisfies H\"ormander's condition \eqref{eq:Hormcond} of arbitrary order.

Let $(\eta_0^{\otimes})(\la)=\eta_0(\la_1)\eta_0(\la_2),$ $\la \in[0,\infty)^2.$ To justify \eqref{eq:discLapTjest} for $T_{1,1}$ we set
\begin{align*}
m_{1,1}^s(\la)&=\frac{|\la_1(1-\la_2^2/4)^{1/2}+\la_2(1-\la_1^2/4)^{1/2}|^{2s}}{\la_1^{2s} }\phi(\la_2/\la_1)(\eta_0^{\otimes})(\la),\\
\tilde{m}_{1,1}^s(\la)&=\frac{|\la_1(1-\la_2^2/4)^{1/2}+\la_2(1-\la_1^2/4)^{1/2}|^{2s}}{\la_2^{2s} }\phi(\la_1/\la_2)(\eta_0^{\otimes})(\la). \end{align*}
Then, using \eqref{eq:discLapMultDef} (in the case $d=1$) we rewrite $T_{1,1}$ as
\begin{equation*}
T_{1,1}(f,g)=B_{m_{1,1}^s}(H(-\dLZ)^s f,Hg)+B_{\tilde{m}_{1,1}^s}( Hf,H(-\dLZ)^sg).
\end{equation*}
In view of Lemma \ref{lem:discLapSpecMult}, to demonstrate \eqref{eq:discLapTjest} it suffices to show that
\begin{equation*}
\|B_{m_{1,1}^s}(f,g)\|_{p}+\|B_{\tilde{m}_{1,1}^s}(f,g)\|_{p}\leq C \|f\|_{p_1}\|g\|_{p_2}.\end{equation*}
This, however, follows directly from Theorem \ref{thm:discLap_Coif-Mey}, since, for each $s>0,$ the multipliers $m_{1,1}^{s},$ and $\tilde{m}_{1,1}^{s},$ satisfy H\"ormander's condition \eqref{eq:Hormcond} of arbitrary order.

Finally, we prove \eqref{eq:discLapTjest} for $T_{1,-1}.$
For $\Real(z)\geq 0$ we set
\begin{align*}
m_{1,-1}^{z}(\la)&=\frac{|\la_1(1-\la_2^2/4)^{1/2}-\la_2(1-\la_1^2/4)^{1/2}|^{2z}}{\la_1^{2z} }\phi(\la_2/\la_1)(\eta_0^{\otimes})(\la),\\
\tilde{m}_{1,-1}^{z}(\la)&=\frac{|\la_1(1-\la_2^2/4)^{1/2}-\la_2(1-\la_1^2/4)^{1/2}|^{2z}}{\la_2^{2z} }\phi(\la_1/\la_2)(\eta_0^{\otimes})(\la). \end{align*}
Then using \eqref{eq:discLapMultDef} (in the case $d=1$) we rewrite $T_{1,-1}$ as
\begin{equation*}
T_{1,-1}(f,g)=B_{m_{1,-1}^{s}}(H(-\dLZ)^s f,(I-H)g)+B_{\tilde{m}_{1,-1}^{s}}( Hf,(I-H)(-\dLZ)^sg).
\end{equation*}
Note that $\mA$ is preserved by $(-\dLZ)^{s}.$ Thus, by Lemma \ref{lem:discLapSpecMult}, to demonstrate \eqref{eq:discLapTjest} it is enough to prove, for $f,g\in \mA,$  the bounds
\begin{equation}
\label{eq:discLapT4decom}
\begin{split}
\|B_{m_{1,-1}^{s}}(H f,(I-H)g)\|_{p}&\leq C \| H f\|_{p_1}\|(I-H)g\|_{p_2},\\
\|B_{\tilde{m}_{1,-1}^{s}}(H f,(I-H)g)\|_{p}&\leq C \|Hf\|_{p_1}\| (I-H)g\|_{p_2}.
\end{split}\end{equation}
We focus only on the first estimate, the reasoning for the second being analogous. We are going to apply Stein's complex interpolation theorem \cite{St1} for each fixed $f\in \mA$. The argument used here takes ideas from the proof of \cite[Theorem 1.4]{GuiKon1}. For further reference we note that the formula
\begin{equation}
\label{eq:BmzdiscLap}
\begin{split}
&B_{m_{1,-1}^z}(H f,(I-H)g)(n)
=\int_0^{1/2}\int_{-1/2}^{0}\phi\bigg(\frac{|\sin \pi \xi_2|}{|\sin \pi \xi_1|}\bigg)\eta_0(|\sin \pi \xi_1|)\eta_0(|\sin \pi \xi_2|)\\
&\times\frac{|\sin\pi \xi_1\sqrt{1-\sin^2 \pi \xi_2}-\sin\pi \xi_2\sqrt{1-\sin^2 \pi \xi_1}|^{2z}}{|\sin \pi \xi_1|^{2z}} \,e^{2\pi i (\xi_1+\xi_2)n}\,\mFZ(f)(\xi_1)\mFZ(g)(\xi_2)\,d\xi;
\end{split}\end{equation}
makes sense not only for $f,g\in\mA$ but more generally, for $f, g\in \ell^2.$ 

 Let $n$ be an even integer larger than $8.$ Then the multipliers $m_{1,-1}^{n+iv},$ $v\in \bR,$ satisfy the Mikhlin-H\"ormander condition \eqref{eq:Hormcond} of order $8$. Thus, Theorem \ref{thm:discLap_Coif-Mey} (with $d=1$) gives
\begin{equation*}
\|B_{m_{1,-1}^{n+iv}}( H f,(I-H)g)\|_{p}\leq C(1+|v|)^{8} \|Hf\|_{p_1}\|(I-H)g\|_{p_2},\qquad v\in \bR.
\end{equation*}
Now, Lemma \ref{lem:DiscFracPF} applied to $\varphi(\la)=\phi(\la_2/\la_1)\eta_0^{\otimes}(\la),$ $\la\in (0,\infty)^2,$ implies
$$B_{m_{1,-1}^{iv}}(H f,(I-H)g)=(-\dLZ)^{iv}\big[B_{\phi(\la_2/\la_1)\eta_0^{\otimes}}(H (-\dLZ)^{-iv}f,(I-H)g)\big].$$ By \cite[Theorem 1.1]{Alex_1} we have $\|(-\dLZ)^{iv}\|_{\ell^q\to\ell^q}\leq C_q (1+|v|)^4,$ $1<q<\infty.$ Hence, Theorem \ref{thm:discLap_Coif-Mey} applied to the multiplier $\phi(\la_1/\la_2)\eta_0^{\otimes}$ produces
\begin{equation*}
\|B_{m_{1,-1}^{iv}}(H f,(I-H)g)\|_{p}\leq C(1+|v|)^{8} \|H f\|_{p_1}\|(I-H)g\|_{p_2},\qquad v\in \bR.
\end{equation*}

By \eqref{eq:BmzdiscLap}, for fixed $f\in \mA,$ the family $\{B_{m_{1,-1}^z}(H f,(I-H)g)\}_{\Real(z)>0}$ consists of analytic operators. This family has admissible growth, more precisely, for each finitely supported $g,h$ we have
$$\big|\langle B_{m_{1,-1}^z}(H f,(I-H)g),h\rangle_{l^2(\bZ)}\big|\leq C_{f,g,h},\qquad |\Real(z)|\leq s.$$
  Consequently, an application of Stein's complex interpolation theorem is permitted and leads to the first inequality in \eqref{eq:discLapT4decom}. The proof of the corollary is thus finished.
\end{proof}

\section{Bilinear radial multipliers for the generic Dunkl transform}
\label{sec:Dun}

Here we apply Theorem \ref{thm:genCoif-Mey} for bilinear multiplier operators associated with the generic Dunkl transform. In the case when the underlying group of reflections is isomorphic to $\bZ_2$ we also prove a fractional Leibniz rule.

Let $R$ be a root system in $\mathbb{R}^d$ and $G$ the associated reflection group (see \cite[Chapter 2]{Ros}). Let $\sigma_{\alpha}(x)$ denote the reflection of $x$ in the hyper-plane orthogonal to $\alpha\in \bR^d$ and let $\ka$ be a nonnegative, $G$ invariant function on $R.$ The differential-difference (rational) Dunkl operators, are defined as
\begin{equation} \nonumber \Do_j f(x)=\partial_j f(x)+ \sum_{\alpha\in R_+}\alpha_j \ka(\alpha)\frac{f(x)-f(\sigma_{\alpha}(x))}{\langle\alpha,x\rangle},\qquad j=1,\ldots,d.\end{equation}
 Here $f$ is a Schwartz function, $R_+$ is a fixed positive subsystem of $R$ and $\langle x,y\rangle=\sum_{j=1}^d x_j y_j$ is the standard inner product. The fundamental property of the operators $\Do_j$ is that, similarly to the usual partial derivatives (which appear when we take $\ka\equiv0$), they commute, i.e. $\Do_l \Do_j=\Do_j \Do_l,$ $l,j=1,\ldots,d.$ The operators $\Do_j$ are also symmetric on $L^2=L^2(\mathbb{R}^d,w(x)dx),$ with $w(x)=w_{\ka}(x):=\prod_{i=1}^d\left|\langle\alpha,x\rangle\right|^{2\ka(\alpha)}.$ Moreover they leave $\mS(\RR^d)$ invariant. Additionally  the Leibniz rule
\begin{equation}
\label{eq:DunkLeib}\Do_j(f_1f_2)(x)=\Do_j(f_1)(x)f_2(x)+\Do_j(f_1)(x)f_2(x),\qquad x\in \Rd,\end{equation}
holds under the extra assumption that one of the functions $f_1,f_2$ is invariant under $G.$

The easiest case of Dunkl operators arrises when $G\sim \ZZ_2^d.$ In other words $G$ consists of reflections through the coordinate axes. In this case
\begin{equation} \nonumber \Do_j f(x)=\partial_j f(x)+  \ka_j\frac{f(x)-f(\sigma_{j}(x))}{x_j},\qquad j=1,\ldots,d,\end{equation}
where $\kappa_j\ge 0,$ while $\sigma_{j}(x)$ denotes the reflection of $x$ in the hyperplane orthogonal to the $j$-th coordinate vector. In this case the weight $w_{\kappa}(x)$ takes the product form $w_{\kappa}(x)=\prod_{j=1}^{d}w_{\kappa_j}(x_j),$ $x\in \RR^d.$

In the (general) Dunkl setting there is an analogue of the Fourier transform, called the Dunkl transform. It is defined by $$\mD f(\xi)=c_{\ka} \int_{\mathbb{R}^d}E(-i\xi,x)f(x)w_{\ka}(x)\,dx$$
 where $E(z,w)=E_{\ka}(z,w)=E_{\ka}(w,z)$ is the so called Dunkl kernel. A defining property of this kernel is the equation
 \begin{equation}\label{eq:DunkkerEq}\Do_{j,x}(E_{\ka}(i\xi,x))=i\xi_j E_{\ka}(i\xi,x),\qquad x\in \RR^d.\end{equation}
The operator $\mD$ has properties similar to the Fourier transform. Namely, we have the Plancherel formula
\begin{equation}
\label{eq:DunklPlan}
\int_{\RR^d} f(x)g(x)\,w(x)\,dx= c_{\kappa}\,\int_{\RR^d}\mD(f)(\xi)\mD(h)(\xi)\, w(\xi)\, d\xi,
\end{equation}
and the inversion formula,
\begin{equation}
 \label{eq:DunklInv}
f(x)=\mD^2 f(-x)=c\int_{\Rd} \mD(f)(\xi)E(i\xi,x)\,w(\xi)\,d\xi,\qquad f\in \mS(\Rd).
\end{equation}

Additionally, the Dunkl transform diagonalizes simultaneously the Dunkl operators $\Do_i,$ i.e. \begin{equation}
\label{eq:Dunkldiagon}\Do_j \mD f=-\mD(i x_j f),\qquad \mD \Do_j f =i \xi_j \mD.\end{equation}

 The Dunkl Laplacian is given by $\Delta_{\ka}=\sum_{i=1}^d \Do_i^2.$ Using the identity
$$
\mD(\Delta_{\ka} f)(\xi)=-|\xi|^2\,\mD(f)(\xi),\qquad \xi \in \Rd, $$
the  operator $-\Delta_{\ka}$ may be formally defined as a non-negative self-adjoint operator on $L^2(\RR^d,w).$ The same is true for $L:=(-\Delta_{\ka})^{1/2}.$ Then, for a bounded function $\mu$ the spectral multiplier $\mu(L)$ is uniquely determined on $\mS(\Rd)$ by
\begin{equation}
 \label{eq:DunSpecMult}
\mD(\mu(L)f)(\xi)=\mu(|\xi|)\mD(f)(\xi) \qquad \xi \in \Rd.
\end{equation}

Consider now $L_1:=L\otimes I$ and $L_2=I\otimes L.$ Analogously to the case of bilinear Fourier multipliers the formula \eqref{eq:genmult} can given by the Dunkl transform. Namely, for a bounded function $m\colon [0,\infty)^2\to \bC$ we have
\begin{equation}\label{eq:defDunklBM}
\begin{split}
&B_m(f_1,f_2)(x)\\
&=\int_{\Rd}\int_{\Rd}\,m(|\xi_1|,|\xi_2|)\,\mD(f)(\xi_1)\,\mD(g)(\xi_2)\,E(i\xi_1,x)E(i\xi_2,x)\,w(\xi_1)w(\xi_2)d\xi_1d\xi_2.
\end{split}
\end{equation}
The above formula is valid pointwise e.g. for Schwartz functions $f_1$ and $f_2$ on $\bR^d.$ We observe that in this section the space $\mA_2$ from \eqref{eq:Aspaces} is
\begin{equation}
\label{eq:DunklA2}
\mA_2=\{g\in L^2(\RR^d,w_{\kappa})\colon \textrm{there is $\varepsilon>0$ such that } \mD(g)(\xi)=0\textrm{ for $|\xi|\not \in [\varepsilon,\varepsilon^{-1}]$}\}.\end{equation}
Thus, by \eqref{eq:Dunkldiagon} the Dunkl derivatives $\delta_j,$ $j=1,\ldots,d,$ preserve $\mA_2.$ 

In this section we will heavily rely on the concepts of Dunkl translation and Dunkl convolution. For $x,y \in \Rd$ The Dunkl translation is defined by
$$\tau^yf(x)=c_{\kappa}\int_{\Rd}\mD(f)(\xi)E(i\xi,x)E(i\xi,y)\,w(\xi)\,d\xi.$$
The inversion formula \eqref{eq:DunklInv} and the properties of the Dunkl kernel imply
$$\mD(\tau^yf)(\xi)=E(-i\xi,y)\mD(f)(\xi).$$ For $f,g\in \mA$ the Dunkl convolution is
$$f\dc g(x)=\int_{\Rd}f(y)\,\tau_{x}\check{g}(y)\,w(y)\,dy,$$
where $\check{g}(x)=g(-x).$ It is known that the Dunkl transform turns this convolution into multiplication, i.e.
\begin{equation}
\label{eq:DunConvFor}
\mD(f \dc g)(x)=\mD(f)(x)\,\mD(g)(x),\qquad [\mD(f) \dc \mD(g)](x)=\mD(f g)(x),\qquad f,g\in \mA.
\end{equation}

The first result of this section is the following Coifman-Meyer type theorem. In what follows we set $\la_{\ka}=(d-1)/2+\sum_{\alpha\in R^+}\ka(\alpha)$ and for brevity write $L^p:=L^p(\Rd,w_{\ka})$ and $\|\cdot\|_{p}=\|\cdot\|_{L^p}.$
\begin{thm}
\label{thm:DunklLap_Coif-Mey} Assume that $m$ satisfies the Mikhlin-H\"ormander condition \eqref{eq:Hormcond} of an order $s> 2\lambda_{\ka}+6$. Then the bilinear multiplier operator given by \eqref{eq:defDunklBM} is bounded from $L^{p_1}\times L^{p_2}$ to $L^p,$ where $1/p_1+1/p_2=1/p,$ and $p_1,p_2,p>1.$ Moreover, the bound \eqref{eq:thm_gen_bound} holds.
\end{thm}
\begin{proof}
 We are going to apply Theorem \ref{thm:genCoif-Mey}. In order to do so we need to check that its assumptions are satisfied  for the operator $L=(-\Delta_{\ka})^{1/2}$. To see that $L$ is injective on its domain we merely note that $w_{\kappa}(\xi)\,d\xi$ is absolutely continuous with respect to Lebesgue measure. The contractivity condition \eqref{eq:Con} follows from \cite[Theorem 4.8]{Ros} and the subordination method. The assumption \eqref{eq:WD} is straightforward from \eqref{eq:defDunklBM} and the Lebesgue dominated convergence theorem, while \eqref{eq:MH} was proved by Dai and Wang \cite[Theorem 4.1]{DaiWan1} (with arbitrary $\rho>\la_{\ka}+1$).

Thus we are left with verifying the property \eqref{eq:PF}, which we prove with $b=2.$ This will be deduced by using the convolution structure associated with Dunkl operators. Let $\varphi_k$ and $\psi_k,$ be smooth functions such that $\supp\varphi_k\subseteq [0,2^{k-2}]$ and $\supp\psi_k\subseteq [2^{k-1},2^{k+1}].$ Let $\tpsi_k$ be a smooth function equal $1$ on $[2^{k-5},2^{k+5}]$ and vanishing outside of $[2^{k-7},2^{k+7}].$ Taking the Dunkl transform of the both sides of \eqref{eq:PF}  and using \eqref{eq:DunSpecMult} we see that our task is equivalent to proving the formula
$$\mD(\varphi_k(L)(f_1)\psi_k(L)(f_2))=\tpsi_k(|\xi|)\mD(\varphi_k(L)(f_1)\psi_k(L)(f_2)),\qquad \xi \in \Rd.$$
Denote $g_j=\mD(f_j),$ $j=1,2.$ By  \eqref{eq:DunConvFor} and \eqref{eq:DunSpecMult} the equation above is exactly 
$$[(\varphi_k(|\cdot|)g_1)\dc(\psi_k(|\cdot|)g_2)](\xi)=\tpsi_k(|\xi|)[(\varphi_k(|\cdot|)g_1)\dc(\psi_k(|\cdot|)g_2)](\xi),\qquad \xi \in \Rd.$$
By definition of $\tpsi$ to prove the last formula it is enough to show that
\begin{equation}
\label{eq:suppPropDunkl}
 \supp [h_1\dc h_2]\subseteq [2^{k-5},2^{k+5}],
\end{equation}
for any functions $h_1$ supported in $B(0,2^{k-2})$ and $h_2$ supported in $B(0,2^{k+1})\setminus B(0,2^{k-1}).$
Take $|\xi| \not\in [2^{k-5},2^{k+5}]$ and $y\in B(0,2^{k-2}).$ We claim that $\tau^{\xi}\check{h_2}(y)=0.$ This implies \eqref{eq:suppPropDunkl}.

Till the end of the proof we thus focus on proving the claim. Let $\gamma_{\xi,y}$ be the distribution given by $\gamma_{\xi,y}(f)=(\tau^{\xi}f)(y),$ $f\in \mS(\Rd).$ In \cite[Theorem 5.1]{AmAnSif1} Amri, Anker, and Sifi proved that $\gamma_{\xi,y}$ is supported in the spherical shell
$$S_{\xi,y}:=\big\{z\in \Rd\colon ||\xi|-|y||\leq |z|\leq |\xi|+|y|\big\}.$$
Therefore, if we prove that $\supp h_2\cap S_{\xi,y}=\emptyset,$ then  $\tau^{\xi}h_2(y)=0.$
Recall that we have $|\xi| \not\in [2^{k-5},2^{k+5}]$ and $y\in B(0,2^{k-2}).$ Take $z\in S_{\xi,y}$ and consider two possibilities, either $|\xi|<2^{k-5}$ or $|\xi|>2^{k+5}.$ In the first case we obtain
$|z|\leq 2^{k-5}+2^{k-2}<2^{k-1},$ while in the second $|z|\geq |\xi|-|y|\geq 2^{k+5}-2^{k-2}> 2^{k+1}.$ Thus, in both the cases $z\not \in \supp h_2,$ and the proof of \eqref{eq:PF} is completed.

\end{proof}

Theorem \ref{thm:DunklLap_Coif-Mey} is quite far from a general bilinear Dunkl multiplier theorem, i.e. when the multiplier function $m$ is not necessarily radial in each of its variables. However, in the case $d=1$ (and $G\sim \bZ_2$), Theorem \ref{thm:DunklLap_Coif-Mey} implies \cite[Theorem 4.1]{AmGasSif1} by Amri, Gasmi, and Sifi. We slightly abuse the notation and, for $\varphi \colon \bR^2 \to \bC,$ $f_1,f_2\in \mA,$ and $x\in \mathbb{R},$ define
\begin{equation}
\label{eq:defDunklBMext}
B_\varphi(f_1,f_2)(x)=\int_{\RR}\int_{\RR}\,\varphi(\xi)\,\mD(f_1)(\xi_1)\,\mD(f_2)(\xi_2)\,E(i\xi_1,x)E(i\xi_2,x)\,w(\xi_1)w(\xi_2)\,d\xi.
\end{equation}
This will cause no confusion with \eqref{eq:defDunklBM}, as till the end of the present section we only use $B_{\varphi}$ given by \eqref{eq:defDunklBMext}.
\begin{cor}[Theorem 4.1 of \cite{AmGasSif1}]
\label{cor:AmGasSif}
Let $G\sim \bZ_2.$ Assume that $\varphi \colon \bR^2 \to \bC$ satisfies the Mikhlin-H\"ormander condition on $\RR^2$ of an order $s>2\la_{\ka}+6,$  namely
\begin{equation}
 \label{eq:Hormcond2}
\|\varphi\|_{MH(\RR^2,s)}:=\sup_{|\alpha|\leq s}\,\sup_{\xi\in \bR^2}|\xi|^{|\alpha|}|\partial^{\alpha} \varphi(\xi_1,\xi_2)|<\infty.
\end{equation}
Then the bilinear multiplier operator given by \eqref{eq:defDunklBMext} is bounded from $L^{p_1}\times L^{p_2}$ to $L^p,$ where $1/p_1+1/p_2=1/p,$ and $p_1,p_2,p>1.$
\end{cor}
\begin{remark}
	When $\kappa=0$ we recover the Coifman-Meyer multiplier theorem in the Fourier transform setting. 
\end{remark}
\begin{proof}[Proof of Corollary \ref{cor:AmGasSif} (sketch)]
Let $\mP(f)(x)=\mD^{-1}(\chi_{\xi>0})\mD(f)(\xi))(x)$ be the projection onto the positive Dunkl frequencies. The corollary  can be deduced from the boundedness of $\mP$ on all $L^p$ spaces $1<p<\infty.$ 
\end{proof}

For $\Real z\ge0,$ let  $(-\Delta_{\ka})^{z}$ be the complex Dunkl derivative
$$\mD[(-\Delta_{\ka})^z(h)](\xi)=|\xi|^{2z}\mD(h)(\xi),\qquad  \xi \in \Rd.$$ The natural $L^2$ domain of this operator is
$$\Dom_{L^2}((-\Delta_{\ka})^z)=\{h\in L^2 \colon |\xi|^{2\Real z}\mD(h)(\xi)\in L^2\}.$$ 
By Plancherel's formula for the Dunkl transform $(-\Delta_{\ka})^z(h)\in L^2$ for $h\in \mA.$  
The second main result of this section is the following fractional Leibniz rule for $(-\Delta_{\ka})^s,$ in the case $G\sim \bZ_2^d.$
\begin{cor}
	\label{cor:DunklLap_Leib}
	Let $G\sim \bZ_2^d$ and take $1/p=1/p_1+1/p_2,$ with $p,p_1,p_2>1.$ Then, for any $s>0,$ we have
	\begin{equation*}
	\|(-\Delta_{\ka})^{s}(f g)\|_{p}\lesssim \|(-\Delta_{\ka})^{s}(f)\|_{p_1}\|g\|_{p_2}+\|f\|_{p_1}\|(-\Delta_{\ka})^{s}(g)\|_{p_2},
	\end{equation*}
	where $f,g\in \mA$ and at least one of the functions $f$ or $g$ is invariant by $G.$
\end{cor}

Before proving the fractional Leibniz rule we need a lemma which is an analogue of Lemma \ref{lem:DiscFracPF}. Its proof is similar, however a bit more technical. Therefore we give more details.
\begin{lem}
\label{lem:DunklFracPF}
Take $d=1$ and let $G\sim \bZ_2.$ Assume that at least one of the functions $f,g\in \mA$ is $G$-invariant. Take $\Real(z)\geq 0$ and let $\varphi\colon \bR^2\to \bC$ be a bounded function that satisfies the Mikhlin-H\"ormander condition \eqref{eq:defDunklBMext} of order $s>2\la_{\kappa}+6$. Then
\begin{equation*}
(-\Delta_{\ka})^{z}(B_{\varphi}(f,g))(x)=\iint_{\RR^2}\,\varphi(\xi)|\xi_1+\xi_2|^{2z}\,\mD(f)(\xi_1)\,\mD(g)(\xi_2)\,E(i\xi_1,x)E(i\xi_2,x)\,w(\xi_1)w(\xi_2)d\xi,
\end{equation*}
for almost all $x\in \Rd.$ 
\end{lem}
\begin{remark}
	It is not obvious why $B_{\varphi}(f,g)\in \Dom_{L^2}((-\Delta_{\ka})^{z}).$ This is explained in the proof of the lemma.
\end{remark}
\begin{proof}
Since the argument is symmetric in $f$ and $g$  we assume that $f$ is $G$-invariant. Denote $E_G(i\xi_1,x)=|G|^{-1}\sum_{g\in G}E(i\xi_1,gx),$ and observe that $E_G$ is $G$-invariant in $x.$  Then, since both $f$ and $\mD(f)$ are $G$-invariant our task reduces to proving that
\begin{equation}
\label{eq:Claim_DunklLap_Leib1}
(-\Delta_{\ka})^{z}(B_{\varphi}(f,g))(x)=\int_{\RR}\int_{\RR}\,\varphi(\xi)|\xi_1+\xi_2|^{2z}\,\mD(f)(\xi_1)\,\mD(g)(\xi_2)\,E_G(i\xi_1,x)E(i\xi_2,x)\,w(\xi_1)w(\xi_2)d\xi,
\end{equation}
for almost all $x\in \Rd.$

For $z=n\in \NN$ this formula is a direct computation, and follows from the Leibniz rule. Indeed, by \eqref{eq:DunkLeib} and \eqref{eq:DunkkerEq} we have
\begin{align*}\Do(B_{\varphi}(f,g))(x)&=\iint_{\RR^2}\,\varphi(\xi)\mD(f)(\xi_1)\,\mD(g)(\xi_2)\,\Do_x[E_G(i\xi_1,x)E(i\xi_2,x)]\,w(\xi_1)w(\xi_2)d\xi\\
&=\iint_{\RR^2}\,\varphi(\xi)\mD(f)(\xi_1)\,\mD(g)(\xi_2)\,i(\xi_1+\xi_2)\,E_G(i\xi_1,x)E(i\xi_2,x)\,w(\xi_1)w(\xi_2)d\xi,
\end{align*}
the interchange of differentiation and integration being allowed since $f,g\in \mA.$ 
Iterating the above equality $2n$ times we obtain \eqref{eq:Claim_DunklLap_Leib1} for $z=n.$ 

We remark that \eqref{eq:Claim_DunklLap_Leib1} for $z\in \NN$  also explains why does $(-\Delta_{\kappa})^{z}(B_{\varphi}(f,g))$ make sense for general $\Real(z)\ge 0$. Indeed, let $n$ be an integer larger than $\Real(z).$ Then, to prove that $B_{\varphi}(f,g)\in \Dom_{L^2}((-\Delta_{\kappa})^{z})$ it is enough to show that $B_{\varphi}(f,g)\in \Dom_{L^2}((-\Delta_{\kappa})^{n}).$ Now, using \eqref{eq:Claim_DunklLap_Leib1} for $z=n,$ together with the binomial formula and \eqref{eq:Dunkldiagon}, we arrive at
\begin{align*}
&(-\Delta_{\ka})^{n}(B_{\varphi}(f,g))(x)\\
&=\sum_{j=0}^{2n} {{2n}\choose {j}}\int_{\RR}\int_{\RR}\,\varphi(\xi)\,\mD(\delta ^j f)(\xi_1)\,\mD(\delta  ^{2n-j}g)(\xi_2)\,E_G(i\xi_1,x)E(i\xi_2,x)\,w(\xi_1)w(\xi_2)d\xi,
\end{align*}
with $\delta$ being the Dunkl operator on $\RR.$ 
Since $f,g$ belong to $\mA_2$ the same is true for $\delta^j f$ and $\delta^{2n-j} g.$ Thus, an application of Corollary \ref{cor:AmGasSif} proves that $B_{\varphi}(f,g)\in \Dom_{L^2}((-\Delta_{\kappa})^{n}),$ as desired.

We come back to demonstrating \eqref{eq:Claim_DunklLap_Leib1} for general $\Real(z)\ge 0.$ Note first that by a continuity argument it suffices to consider $\Real(z)>0.$ Denoting
$$T_z(f,g)(x)=\iint_{\RR^2}\,\varphi(\xi)|\xi_1+\xi_2|^{2z}\,\mD(f)(\xi_1)\,\mD(g)(\xi_1)\,E(i\xi_1,x)E(i\xi_2,x)\,w(\xi_1)w(\xi_2)d\xi.$$
our task is reduced to proving that
\begin{equation}
\label{eq:DunklFormDual1}
\langle (-\Delta_{\ka})^{z}(B_{\varphi}(f,g)), h \rangle_{L^2}=\langle T_z(f,g), h \rangle_{L^2},
\end{equation}
for $h\in \mA_2\cap \mS(\RR)$ (recall that $\mA_2$ is given by \eqref{eq:DunklA2}).  This is enough because $\mA_2\cap \mS(\RR)$ is dense in $L^2.$ From \eqref{eq:Claim_DunklLap_Leib1} for $z\in \NN$  we deduce that for any polynomial $P$ it holds
\begin{equation}
\label{eq:Claim_DunklLap_Leib1Aux}
\begin{split}
&P(-\Delta_{\ka})(B_{\varphi}(f,g))(x)\\
&=\iint_{\RR^2}\,\varphi(\xi)P(|\xi_1+\xi_2|^2)\,\mD(f)(\xi_1)\,\mD(g)(\xi_2)\,E_G(i\xi_1,x)E(i\xi_2,x)\,w(\xi_1)w(\xi_2)d\xi.
\end{split}
\end{equation}
For brevity we denote by $T^P(f,g)(x)$ the right hand side of \eqref{eq:Claim_DunklLap_Leib1Aux}.
Note that $\mD(f),$ $\mD(g),$ and $\mD(h)$ are supported in $[-N,N]$ for some large $N.$ Let $\{P_{r}(t)\}_{r\in\NN},$ be a sequence of polynomials that converges uniformly to $t^{z}$ on $[0,4N^2].$ Then, \eqref{eq:DunklPlan}, \eqref{eq:Dunkldiagon}, and \eqref{eq:Claim_DunklLap_Leib1Aux} imply
\begin{equation}
\label{eq:DunklPjform}
\begin{split}
&\int_{\RR}\,P_r(|\zeta|^2)\mD(B_{\varphi}(f,g))(\zeta)\,\mD(\bar{h})(\zeta)\,w(\zeta)\,d\zeta=\langle P_r(-\Delta_{\ka})(B_{\varphi}(f,g)),h\rangle_{L^2}\\
&=\langle T^{P_r}(f,g),h\rangle_{L^2}.
\end{split}\end{equation}
Now, since $\supp \mD(\bar{h})\subseteq [-N,N]$ and $\mD(B_{\varphi}(f,g))\,\mD(\bar{h})\in L^1,$ the dominated convergence theorem shows that the left hand side of \eqref{eq:DunklPjform} converges to $\langle (-\Delta_{\ka})^{z}(B_{\varphi}(f,g)), h \rangle_{L^2}$ as $r\to \infty.$ Similarly, since $\mD(f)$ and $\mD(g)$ are supported in $[-N,N]$ the expression $T^{P_r}(f,g)(x)$ is uniformly bounded in $r\in \NN$ and $x\in \RR$ and converges to $T_z(f,g)(x)$ as $r\to \infty.$  As $h\in \mS(\RR)$ the dominated convergence theorem implies $\lim_{r\to \infty}\langle T^{P_r}(f,g),h\rangle_{L^2} = \langle T_z(f,g), h\rangle_{L^2}.$ Therefore, \eqref{eq:DunklFormDual1} is justified and hence, also \eqref{eq:Claim_DunklLap_Leib1}. This completes the proof of Lemma \ref{lem:DunklFracPF}.
\end{proof}

We now pass to the proof of Corollary \ref{cor:DunklLap_Leib}.
\begin{proof}
By repeating the argument from the beginning of the proof of Corollary \ref{cor:LeibDiscLap} (with sums replaced by integrals) our task is reduced to $d=1.$ We devote the present paragraph to a brief justification of this statement Here we need the fact that for $s\ge 0$ and $L_j=-\Do_j^2,$ the operators  $(L_j)^{s}(-\Delta_{\kappa})^{-s}$ as well as $(-\Delta_{\kappa})^{s}(\sum_{j=1}^d(L_j)^s)^{-1}$, are bounded on all $L^p,$ $1<p<\infty.$ This is true by e.g. \cite[Corollary 3.2]{WroJSMMSO}, since in the product setting each $L_j,$ $j=1,\ldots,d,$ generates a symmetric contraction semigroup. Then we are left with showing that
\begin{equation}
\label{eq:LeibDunklLapAux2}
\|L_j^s(fg)\|_p\lesssim \|L_j^{s}f\|_{p_1}\|g\|_{p_2}+ \|L_j^{s}g\|_{p_1}\|f\|_{p_2}
\end{equation}
cf.\ \eqref{eq:LeibDiscLapAux1}. The proof of \eqref{eq:LeibDunklLapAux2} is similar to that of \eqref{eq:LeibDiscLapAux1}, thus we give a sketch when $j=1.$ For $t\in \RR$ and $x\in \RR^{d-1},$  consider the auxiliary functions $f_x(t)=f((t,x))$ and $g_x(t)=g((t,x)).$ Then, setting
$w_{\kappa}^{(1)}(x)=\prod_{i=2}^{d}w_{\kappa_{i}}(x),$ we write 
$$\|L_1^{s}(fg)\|_{p}=\int_{\RR^{d-1}} \|L_1^{s}(f_x(\cdot)g_x(\cdot))\|_{L^p(\RR,w_{\kappa_1})}^{p}\,w_{\kappa}^{(1)}(x)\,dx.$$ 
From this point on we repeat the steps in the proof of \eqref{eq:LeibDiscLapAux1}. Namely, we apply the fractional Leibniz rule for $d=1$ and H\"older's inequality (for integrals). We omit the details here. From now on we focus on proving Corollary \ref{cor:DunklLap_Leib} for $d=1.$

Let $\phi$ be a function in $C^{\infty}([0,\infty))$ supported in $[0,1/4]$ and such that $\phi(t)+\phi(t^{-1})=1.$ Setting
\begin{align*}
T_1(f,g)(x)=\iint_{\RR^2}\,\phi(|\xi_2|/|\xi_1|)|\xi_1+\xi_2|^{2s}\,\mD(f)(\xi_1)\,\mD(g)(\xi_2)\,E(i\xi_1,x)E(i\xi_2,x)\,w(\xi_1)w(\xi_2)d\xi,\\
T_2(f,g)(x)=\iint_{\RR^2}\,\phi(|\xi_1|/|\xi_2|)|\xi_1+\xi_2|^{2s}\,\mD(f)(\xi_1)\,\mD(g)(\xi_2)\,E(i\xi_1,x)E(i\xi_2,x)\,w(\xi_1)w(\xi_2)d\xi. \end{align*}
and using Lemma \ref{lem:DunklFracPF} with $\varphi\equiv 1$  we rewrite \begin{equation*}(-\Delta_{\ka})^{s}(fg)=T_1(f,g)+T_2(f,g).\end{equation*}
From now on the proof resembles that of Corollary \ref{cor:LeibDiscLap} (in fact it is even easier). We need to prove,  for $f,g\in\mA,$ the estimate 
\begin{equation*}
\|T_1(f,g)\|_{p}\leq C \|(-\Delta_{\ka})^{s}f\|_{p_1}\|g\|_{p_2},\qquad \|T_2(f,g)\|_{p}\leq C \|f\|_{p_1}\|(-\Delta_{\ka})^{s}g\|_{p_2}.
\end{equation*}
We focus only on the first inequality, as the proof of the second is analogous. For $\Real(z)\geq 0$ we set
\begin{equation*}m^z(\xi_1,\xi_2)=\frac{|\xi_1+ \xi_2|^{2z}}{|\xi_1|^{2z}}\phi(|\xi_2|/|\xi_1|),\qquad \xi \in \bR^2,\end{equation*}
so that $T_1(f,g)=B_{m^s}((-\Delta_{\kappa})^s f,g).$
Since $\mA$ is preserved under $(-\Delta_{\kappa})^s$ our task is reduced to showing that, for $s>0$ it holds 
\begin{equation}\label{eq:DunklLap_Leib_Aux4}
\|B_{m^s}(f,g)\|_{p}\leq C \|f\|_{p_1}\|g\|_{p_2},\qquad f,g\in \mA
\end{equation}

As in Section \ref{sec:Zd} we are going to apply Stein's complex interpolation theorem. To do this we need to extend $B_{m^z}(f,g)$ outside of $\mA\times \mA,$ by allowing $g$ to be a simple function. This may be achieved by a limiting process. Namely, instead of $m^z$ we consider $m^z_{\e}=m^z e^{-\e|\xi|^2}.$ Then, 
\begin{equation}
\label{eq:DunklLapBmepsDef}
\begin{split}
&B_{m^s_{\e}}(f,g)(x)\\
&:=\int_{\Rd}\int_{\Rd}\,e^{-\e|\xi|^2}\phi(|\xi_2|/|\xi_1|)\,\frac{|\xi_1+\xi_2|^{2z}}{|\xi_1|^{2z}}\,\mD(f)(\xi_1)\,\mD(g)(\xi_2)\,E(i\xi_1,x)E(i\xi_2,x)\,w(\xi_1)w(\xi_2)d\xi
\end{split}
\end{equation}
converges pointwise to $B_{m^s}(f,g)$ as $\varepsilon\to 0^+,$  whenever $f,g \in \mA.$ Therefore, by Fatou's Lemma, to prove \eqref{eq:DunklLap_Leib_Aux4} for $B_{m^s}$ it is enough to prove it for each $B_{m^s_{\e}},$ $\e>0,$ as long as
\begin{equation*}
\|B_{m^s_{\e}}(f,g)\|_{p}\leq C \|f\|_{p_1}\|g\|_{p_2},
\end{equation*}
where $C$ is independent of $\e$. The gain is that now \eqref{eq:DunklLapBmepsDef} is well defined for $g\in L^2,$ in particular it is valid for simple functions.

Let $n>2\la_{\ka}+6.$ Then the multipliers $m^{n+iv}_{\e},$ $j=1,2,$ $v\in \bR,$ satisfy H\"ormander's condition \eqref{eq:Hormcond2} of order $2\la_{\ka}+6$. Thus, using Corollary \ref{cor:AmGasSif} (with $d=1$) we obtain
\begin{equation*}
\|B_{m^{n+iv}_{\e}}(f,g)\|_{p}\leq C_n(1+|v|)^{2\la_{\ka}+2} \|f_1\|_{p_1}\|f_2\|_{p_2},\qquad v\in \bR.
\end{equation*}
Now, Lemma \ref{lem:DunklFracPF} applied to $\varphi(\xi)=\phi(|\xi_2|/|\xi_1|)e^{-\e(|\xi|^2)}$ implies
$$B_{m^{iv}_{\e}}(f,g)=(-\Delta_{\ka})^{iv}\big[B_{\varphi}((-\Delta_{\ka})^{-iv}f,g)\big].$$ Thus, using \cite[Theorem 4.1]{DaiWan1} followed by Corollary \ref{cor:AmGasSif} (for the multiplier $\phi(|\xi_2|/|\xi_1|)e^{-\e(|\xi|^2)}$) we obtain
\begin{equation*}
\|B_{m^{iv}}(f,g)\|_{p}\leq C(1+|v|)^{2\la_{\ka}+2} \|f_1\|_{p_1}\|f_2\|_{p_2},\qquad v\in \bR.
\end{equation*}
By definition
\begin{align*}&B_{m^{z}_{\e}}(f,g)\\
&
=\iint_{\RR^2}\,\frac{|\xi_1+ \xi_2|^{z}}{|\xi_1|^{z}}\phi(|\xi_2|/|\xi_1|)e^{-\e|\xi|^2}\,\mD(f)(\xi_1)\,\mD(g)(\xi_2)\,E(i\xi_1,x)E(i\xi_2,x)\,w(\xi_1)w(\xi_2)d\xi_1d\xi_2.\\
\end{align*}
Hence, for fixed $f_1\in \mA$ the family $\{B_{m^{z}}(f,g)\}_{\Real(z)>0}$ consists of analytic operators. This family has admissible growth, more precisely, for each simple function $h$ we have
$$\big|\langle B_{m^{z}}((-\Delta_{\ka})^{z}f,g),h\rangle_{L^2}\big|\leq C_{f,g,h,s},\qquad |\Real(z)|\leq s.$$
  Consequently, using Stein's complex interpolation theorem is permitted and leads to \eqref{eq:DunklLap_Leib_Aux4}. The proof of the corollary is thus finished.
\end{proof}

\section{Bilinear multipliers for Jacobi trigonometric polynomials}

\label{sec:Jac}

In this section we give a bilinear multiplier theorem for expansions in terms of Jacobi trigonometric polynomials. Contrary to the previous sections we do not prove a fractional Leibniz rule here. The reason for this is that there is no natural first order operator in the Jacobi setting that satisfies a Leibniz-type rule of integer order.

Let $\alpha,\beta > -1/2$ be fixed, and let $P_n^{\ab}$ be the one-dimensional Jacobi polynomials of type $\ab.$ For $n \in \mathbb{N}$ and $-1<x<1$ these are given by the Rodrigues formula
$$
P_n^{\ab}(x) = \frac{(-1)^k}{2^n n!} (1-x)^{-\alpha}(1+x)^{-\beta} \frac{d^n}{dx^n}
\Big[ (1-x)^{\alpha+k}(1+x)^{\beta+k}\Big].$$
We now substitute $x=\cos \theta,$ $\theta \in [0,\pi],$ and consider
the trigonometric Jacobi polynomials
$
P_n^{\ab}(\cos\theta).
$
This is an orthogonal and complete system in $L^2(d\mu_{\ab})$, where
$$
d\mu_{\ab}(\theta) = \Big( \sin\frac{\theta}2 \Big)^{2\alpha+1} 
\Big( \cos\frac{\theta}2\Big)^{2\beta+1} d\theta.
$$
Throughout this chapter we abbreviate $L^p:=L^p([0,\pi],\mu_{\ab})$ and $\|\cdot\|_p:=\|\cdot\|_{L^p}.$
 Now, setting $\P_n(\theta)=\P_n^{\ab}(\theta)=c_n^{\ab}P_k^{\ab}(\cos\theta),$ where $\|P_n^{\ab}(\cos\cdot)\|_{2}=(c_n^{\ab})^{-1}$ we obtain a complete orthonormal system in $L^2.$ Each $\P_n^{\ab}$ is an eigenfunction of the differential operator 
$$
\J=\J^{\ab} = - \frac{d^2}{d\theta^2} - \frac{\alpha-\beta+(\alpha+\beta+1)\cos\theta}{\sin \theta}
\frac{d}{d\theta} + \Big( \frac{\alpha+\beta+1}{2}\Big)^2;
$$
with the corresponding eigenvalue being $( n + \frac{\alpha+\beta+1}{2})^2$. In what follows we set $\gamma=(\alpha+\beta+1)/2;$ observe that $\gamma>0.$

In this setting the spectral multipliers of $\J^{1/2}$ are given by $$
\mu(\J^{1/2}) f=\sum_{n\in \bN} \mu\big( n + \gamma\big) \lan f,\P_k\ran_{L^2} \P_k.
$$ 
If $\mu\colon \mathbb{R}_+\to \mathbb{C}$ is bounded, then $\mu(\J^{1/2})$ is a bounded operator  on $L^2.$ In this section the formula \eqref{eq:genmult} defining bilinear multipliers becomes
\begin{equation}
\label{eq:Jacmult}
\begin{split}
&B_m(f_1,f_2)(\theta)=m(\J^{1/2}\otimes I,I\otimes \J^{1/2})(x,x)\\
&=\sum_{n_1\in \bN,n_2 \in \bN}m\Big(n_1 + \gamma,n_2 + \gamma\Big)\lan f_1,\P_{n_1}\ran \lan f_2,\P_{n_2}\ran\,\P_{n_1}(\theta)\P_{n_2}(\theta).
\end{split}
\end{equation}
The space $\mA$ from \eqref{eq:Aspaces} coincides with the linear span of $\{\P_n\}_{n\in \bN}.$ We prove the following Coifman-Meyer type multiplier theorem.

\begin{thm}
	\label{thm:Jac_Coif-Mey} Assume that $m$ satisfies H\"ormander's condition \eqref{eq:Hormcond} of order $s>4(\alpha+\beta)+15.$ Then the bilinear multiplier operator given by \eqref{eq:Jacmult} is bounded from $L^{p_1}\times L^{p_2}$ to $L^p,$ where $1/p_1+1/p_2=1/p,$ and $p_1,p_2,p>1.$ Moreover, the bound \eqref{eq:thm_gen_bound} is valid.
\end{thm}
\begin{remark}
	The theorem implies a Coifman-Meyer type multiplier result for bilinear multipliers associated with the modified Hankel transform. This follows from a transference results of Sato \cite{Sato1}.
\end{remark}
\begin{proof}
Once again the proof hinges on Theorem \ref{thm:genCoif-Mey}. We need to verify that $L=\J^{1/2}$ satisfies its assumptions. The injectivity condition is clear since $0$ is not an eigenvalue of $\J^{1/2}.$ The contractivity assumption \eqref{eq:Con} can be inferred from the formula $$e^{-t\J}(f\circ \cos)(\theta)=e^{-t(\alpha+\beta+1)^2/4}T_t^{\alpha,\beta}f(\cos \theta)$$ relating the semigroup $e^{-t\J}$ with the semigroup $T_t^{\alpha,\beta}$ from the Jacobi polynomial setting, as $T_t^{\alpha,\beta}$ is well known to be Markovian. The condition \eqref{eq:WD} is straightforward, since $\mA$ is the linear span of Jacobi trigonometric polynomials.  The Mikhlin-H\"ormander functional calculus \eqref{eq:MH} for $\J^{1/2}$ (with $\rho=2\alpha+2\beta+13/2$) was obtained in \cite[Corollary 4.3]{WroMSM}.

It remains to show \eqref{eq:PF}. Here we need the following identity
\begin{equation}
\label{eq:PolProd}
\P_{n_1}(\theta)\P_{n_2}(\theta)=\sum_{j=|n_1-n_2|}^{j=n_1+n_2}c_{n_1,n_2}(j)\,\P_j(\theta).
\end{equation}
The above is well known to hold for general orthogonal polynomials on an interval contained in $\RR$, hence also for $\P_j$ as they are merely a reparametrisation of the Jacobi polynomials. 

We prove that \eqref{eq:PF} holds with $b=3.$ Take $f,g\in \mA.$ Then $$f_1=\sum_{n_1\in \NN} c_{n_1}^1\P_{n_1},\qquad f_2=\sum_{n_2\in \NN} c_{n_2}^2\P_{n_2},$$ where all but a finite number of $c_n^1,$ $c_n^2$ vanish. Denote
$$R_{a,b}=\{n\in \bN \colon 2^{a}-\gamma\le n\le 2^{b}-\gamma\}.$$ 
Since $\varphi_k$ and $\psi_k$ are supported in $[0,2^{k-3}]$ and $[2^{k-1},2^{k+1}],$ respectively, we have
$$\varphi_k(L)(f_1)=\sum_{n_1\in \NN\colon n_1+\gamma\le 2^{k-3}}c_{n_1}^1\, \varphi_k(n_1+ \gamma)\,\P_{n_1}$$
whereas
$$\psi_k(L)(f_2)=\sum_{n_2\in R_{k-1,k+1}}c_{n_2}^2\, \psi_k(n_2+ \gamma)\,\P_{n_2}.$$

Now, if $n_1+\gamma \le 2^{k-3}$ and $2^{k-1}\le n_2+\gamma \le 2^{k+1},$ then we must also have $$|n_1-n_2|\ge 2^{k-1}-2^{k-3}\ge 2^{k-2}\quad\textrm{and}\quad  n_1+n_2 \le 2^{k-3}-\gamma +2^{k+1}-\gamma \le 2^{k+2}-2\gamma.$$
Since $\gamma>0,$ we see that if $|n_1-n_2|\le n\le n_1+n_2,$ then $2^{k-2}\le n+\gamma\le 2^{k+2} .$ Consequently, in view of \eqref{eq:PolProd}, the operator $\tpsi_k(L)$ leaves invariant each product $\P_{n_1}\cdot \P_{n_2},$ hence, also $\varphi_k(L)(f_1)\cdot \psi_k(L)(f_2).$ The proof of \eqref{eq:PF} is thus completed.

\end{proof}

\subsection*{Acknowledgments}
I thank prof.\ Krzysztof Stempak for suggesting the idea to combine joint spectral multipliers with multilinear multipliers, prof.\ Christoph Thiele for a discussion on the Coifman-Meyer multiplier theorem and useful remarks during the preparation of the paper, prof.\ Camil Muscalu for a discussion on multilinear multipliers, prof.\ Herbert Koch, prof.\ Fulvio Ricci, and dr Gian Maria Dall'Ara for discussions on the fractional Leibniz rule, dr hab.\ Wojciech M\l otkowski for bringing to my attention the formula \eqref{eq:PolProd}, and dr Jotsaroop Kaur for discussions on Hermite polynomials.

Part of the research presented in this paper was carried over while the author was \textit{Assegnista di ricerca}
at the Universit\`a di Milano-Bicocca. The research was supported by Italian PRIN 2010 ``Real and complex manifolds:
geometry, topology and harmonic analysis"; Polish funds for sciences, National Science Centre (NCN), Poland, Research Project 2014\slash 15\slash D\slash ST1\slash 00405; and by the Foundation for Polish Science START Scholarship.

\end{document}